\documentclass[twoside]{jT}
\usepackage{amssymb} 
\usepackage{amsmath} 
\usepackage{latexsym}
\usepackage{graphicx} 
\usepackage{enumerate}
\newtheorem{theorem}{Theorem}[section]
\newtheorem{lemma}{Lemma}[section]
\newtheorem{proposition}{Proposition}[section]
\newtheorem{corollary}{Corollary}[section]

\theoremstyle{definition}
\newtheorem*{definition}{Definition}
\newtheorem*{remark}{Remark}
\newtheorem*{example}{Example}

\numberwithin{equation}{section}

\newtheorem*{conjecture}{Conjecture}

\def\rb{\mathbb{R}}
\def\qb{\mathbb{Q}}

\def\cb{\mathbb{C}}
\def\nb{\mathbb{N}}
\def\zb{\mathbb{Z}}

\def\uP{\underline{P}}
\def\uT{\underline{T}}

\begin{document}
\title[On Salem numbers, expansive polynomials and Stieltjes continued fractions]{On Salem numbers, expansive polynomials and Stieltjes continued fractions 
}

\author{\sc Christelle Guichard}
\address{Christelle Guichard\\
Institut Fourier, CNRS UMR 5582, \\
Universit\'e de Grenoble I, \\
BP 74, Domaine Universitaire,\\
38402 Saint-Martin d'H\`eres, France}
\email{ Christelle.Guichard@ujf-grenoble.fr}
\urladdr{http://www-fourier.ujf-grenoble.fr}

\author{\sc Jean-Louis Verger-Gaugry}
\address{Jean-Louis Verger-Gaugry\\
Institut Fourier, CNRS UMR 5582, \\
Universit\'e de Grenoble I, \\
BP 74, Domaine Universitaire,\\
38402 Saint-Martin d'H\`eres, France}
\email{jlverger@ujf-grenoble.fr}
\urladdr{http://www-fourier.ujf-grenoble.fr}


\maketitle


\begin{resume}
Dans cet article on montre que pour tout polyn\^ome $T$, de degr\'e $m \geq 4$,
\`a racines simples sans racine dans $\{\pm 1\}$, qui est soit de Salem soit cyclotomique,
il existe un polyn\^ome expansif unitaire $P(z) \in \zb[z]$ tel que
$(z-1) T(z) = z P(z) - P^{*}(z)$. Cette \'equation d'association
utilise le Th\'eor\`eme A (1995) de Bertin-Boyd d'entrecroisement de conjugu\'es sur le cercle unit\'e.
L'ensemble des polyn\^omes expansifs unitaires
$P$ qui satisfont cette \'equation d'association contient
un semi-groupe commutatif infini.
Pour tout $P$ dans cet ensemble, caract\'eris\'e par un certain crit\`ere,
un nombre de Salem est produit et cod\'e par un $m$-uplet de nombres rationnels strictement positifs
caract\'erisant la fraction conti-nue de Stieltjes (SITZ)
du quotient (alternant) d'Hurwitz corres-pondant \`a $P$.
Ce codage est une r\'eciproque \`a la Construction  de Salem (1945).
La structure de semi-groupe se transporte sur
des sous-ensembles de fractions continues de Stieltjes, ainsi que
sur des sous-ensembles de nombres de Garsia
g\'en\'eralis\'es.   
\end{resume}

\begin{abstr}
In this paper we show that for every Salem polynomial or cyclotomic polynomial,
having simple roots and no root in $\{\pm 1\}$, denoted by $T$, deg $T = m \geq 4$,
there exists a monic expansive polynomial $P(z) \in \zb[z]$ such that
$(z-1) T(z) = z P(z) - P^{*}(z)$. 
This association equation makes use of 
Bertin-Boyd's Theorem A (1995) of interlacing of conjugates on the unit circle.
The set of monic expansive polynomials $P$ satisfying this
association equation 
contains an infinite commutative semigroup.
For any $P$ in this set, characterized by a certain criterion,
a Salem number $\beta$ is produced and coded
by an $m$-tuple of positive rational numbers characterizing  
the (SITZ) Stieltjes continued fraction of the corresponding Hurwitz quotient (alternant)
of $P$. This coding is a converse method to the Construction of Salem (1945).
Subsets of Stieltjes continued fractions, and
subsets of generalized Garsia numbers, inherit this semigroup structure.
\end{abstr}

\vspace{0.5cm}

\tableofcontents

\bigskip

\section{Introduction}
\label{S1}

A Salem number is an algebraic integer $\theta > 1$ such that the Galois conjugates
$\theta^{(i)}$ of $\theta$ satisfy: $|\theta^{(i)}| \leq 1$ with at least one conjugate 
of modulus 1. The set of Salem numbers is denoted by ${\rm T}$.
A Pisot number is a real algebraic integer, all of whose other conjugates 
have modulus strictly less than 1. The set of Pisot numbers is traditionally
denoted by S. 

An open problem is the characterization of the set $\overline{{\rm T}}$
of limit points of T. A first Conjecture of Boyd \cite{boyd00}
asserts that the union S $ \cup$ T is closed. A second Conjecture of Boyd
(\cite{boyd0}, p 327) asserts that the first derived set of S $ \cup$ T is
S. At least, S is included in $\overline{{\rm T}}$ (\cite{bertinetal}, Theorem 6.4.1).
In 1945, Salem (\cite{salem}, Theorem IV) developped the so-called {\it Construction of Salem}
to show that every Pisot number is a limit point of convergent sequences of Salem numbers
from both sides. 
Siegel \cite{siegel} proved that the smallest Pisot number is $\theta_0 = 1.32\ldots$,
dominant root of $X^3 - X -1$, implying that the number of Salem numbers 
in any interval $(1,M)$, with $M \geq \theta_0$, is infinite. However,
if 1 were a limit point of T, then T
would be everywhere dense in $[1, +\infty)$ since $\theta \in {\rm T}$
implies $\theta^m \in {\rm T}$ for all positive integer $m$ 
by \cite{salem0}, p. 106, 
\cite{pisot}, p. 46. That T is everywhere dense in
$[1,+\infty)$ is probably false \cite{boyd00}; even though the 
Conjecture of Lehmer were true for Salem numbers, 
the open problem of the existence of a smallest Salem number $> 1$
would remain to be solved.
Converse methods to the Construction of Salem 
to describe interesting sequences of algebraic numbers, eventually Salem numbers,
converging to a given Salem number
are more difficult to establish.

{\it Association theorems} between Pisot polynomials and Salem polynomials, 
which generically make use of the polynomial relation 
\begin{center}
 $(X^2 + 1) P_{Salem} = X P_{Pisot}(X) + P^{*}_{Pisot}(X),$
\end{center}
were introduced by Boyd (\cite{boyd0} Theorem 4.1), 
Bertin and Pathiaux-Delefosse (\cite{bertinboyd}, 
\cite{bertinpathiauxdelefosse} pp 37--46, \cite{bertinetal} chapter 6), 
to investigate the links between infinite collections of Pisot numbers 
and a given Salem number.

In the scope of studying interlacing on the unit circle and
$\overline{{\rm T}}$, McKee and Smyth \cite{mckeesmyth} 
recently used new interlacing theorems, different from
those introduced (namely Theorem A and Theorem B) by Bertin and Boyd 
\cite{bertinboyd} \cite{bertinpathiauxdelefosse}.
These interlacing theorems, and their limit-interlacing versions, 
turn out to be fruitful. 
Theorem 5.3 (in \cite{mckeesmyth}) shows that all Pisot numbers
are produced by a suitable (SS) interlacing condition, 
supporting the second Conjecture of Boyd; 
similarly Theorem 7.3 (in \cite{mckeesmyth}), using Boyd's association theorems, 
shows that all Salem numbers are produced by interlacing and
that a classification of Salem numbers
can be made.

In the present note, we reconsider the interest of 
the interlacing Theorems  of \cite{bertinboyd}, as
potential tools for this study of limit points, 
as alternate analogues of those of McKee and Smyth;
we focus in particular on Theorem A of \cite{bertinboyd}
(recalled as Theorem \ref{theoremA} below).

The starting point is the observation 
that expansive polynomials are basic ingredients
in the proof of Theorem A and that the theory of
expansive polynomials recently, and independently, received 
a strong impulse from Burcsi \cite{burcsi}.
The idea is then to bring back to Theorem A a certain number of 
results from this new theory:
Hurwitz polynomials, alternants, their coding in continued fractions...
In section \ref{S2} Theorem A and McKee and Smyth's interlacing 
modes of conjugates
on the unit circle are recalled (to fix the notation);
association theorems between
Salem polynomials and expansive polynomials 
are obtained. In particular, the main theorem we prove
is the following.

\begin{theorem}
\label{main1}
Let $T$ in $\zb[z]$, with simple roots 
$\not\in \{\pm 1\}$ of degree $m \geq 4$,
be either a cyclotomic polynomial or a Salem polynomial. Then
there exists a monic expansive polynomial
$P(z)  \in \zb[z]$ of degree $m$ such that
\begin{equation}
\label{main1eq}
(z-1) T(z) = z P(z) - P^{*}(z).
\end{equation}
\end{theorem}
 
In fact, the set $\mathcal{P}_{T}$ (resp. $\mathcal{P}_{T}^{+}$) 
of monic expansive polynomials
$P(z)  \in \zb[z]$ of degree $m$ which satisfy
\eqref{main1eq} (resp. of polynomials 
$P  \in \mathcal{P}_{T}$
having positive constant coefficient) 
is proved to be infinite in \S \ref{sec:SS4.2S2}
and \S \ref{sec:SS4.3S2}.
Moreover we prove that 
$\mathcal{P}_{T}^{+} \cup \{T\}$ has 
a commutative semigroup structure with internal law:
\begin{equation}
\label{loiPP} 
(P, P^{\dag}) ~\to~ P \oplus P^{\dag} ~:=~ P + P^{\dag} - T,
\end{equation}
where $\{T\}$ is the neutral element. 
We prove Theorem \ref{caseTWO}, as analogue of 
Theorem \ref{main1} for the existence of
expansive polynomials with negative constant terms. 
In section \ref{S3} only Salem polynomials $T$ are considered; 
a Stieltjes continued fraction is
shown to code analytically not only a Salem number 
but also all its interlacing conjugates, using Hurwitz polynomials.
The second main theorem is concerned with the algebraic structure
of this coding, as follows.

\begin{theorem} 
\label{main2}
Let $T \in \zb[z]$ be a Salem polynomial with simple roots
$\not\in \{\pm 1\}$ of degree $m$.
For each monic expansive polynomial
$P(z)  \in \mathcal{P}_{T}$, denote
\begin{equation}
\label{main2eq} 
[f_1 / f_2 / \ldots / f_m](z) = 
\cfrac{f_1}{ 1+ \cfrac {f_2 z}{ 1+\cfrac{f_3 z}{ 1+ \cfrac{...}{1+f_m z}}}}
\end{equation}
the Hurwitz alternant $h_{P}(z)$ uniquely associated with $P$, written
as a Stieltjes continued fraction. 
Let $\mathcal{F}_{T}$, resp. $\mathcal{F}_{T}^{+}$, be the set of 
all $m$-tuples 
$\mbox{}^{t}(f_1, f_2, \ldots, f_m) \in (\qb_{> 0})^m$
such that 
$[f_1 / f_2 / \ldots / f_m](z) = h_{P}(z)$ for $P$ running over
$\mathcal{P}_{T}$, resp. $\mathcal{P}_{T}^{+}$.
Then 
\begin{itemize}
\item[(i)] $\mathcal{F}_{T}$ is discrete and
has no accumulation point in $\bigl(\rb_{> 0}\bigr)^m$,
\item[(ii)] $\mathcal{F}_{T}$ has affine dimension
equal to $m$,
\item[(iii)] the subset $\mathcal{F}_{T}^{+} \cup \{0\}$ is 
a commutative semigroup, whose internal law
is the image of \eqref{loiPP} by the mapping $P \to h_{P}$,
with $\{0\}$ as neutral element, 
\item[(iv)] the intersection of $\mathcal{F}_{T}$ with 
the hypersurface
defined by $$\{\mbox{}^{t}(x_1, \ldots, x_m) \in \rb^m \mid [x_1 / x_2 / \ldots / x_m](1) = 1\}$$
is empty.
\end{itemize}
\end{theorem}
 
To a Salem number $\beta$ and a Salem polynomial $T$ vanishing at $\beta$,
this coding associates the point set $\mathcal{F}_{T}$.
The arithmetico-analytic deformation and 
limit properties of the inverse of this coding function (as a 
converse of the `Construction of Salem') will be reported elsewhere.
By this new approach we hope to shed some light on the existence of 
very small nonempty open intervals in the neighbourhood of
$\beta$ deprived of any Salem number. The interest 
lies in the following Lemma (Appendix).

\begin{lemma}
\label{intervals_nonexistence}
If there exists an (nonempty) open interval of $(1,+\infty)$ which does not
contain any Salem number, then the Conjecture of Lehmer for Salem numbers
is true.
\end{lemma}

The theory of expansive polynomials is fairly recent
\cite{burcsi}.
The terminology ``expansive polynomial'' appeared in
the study of canonical number systems (CNS), for instance in Kov\'acs
\cite{kovacs} and in Akiyama and Gjini \cite{akiyamagjini} for self-affine attractors in
$\rb^{n}$. Then expansive polynomials were associated canonically with
Hurwitz polynomials by Burcsi \cite{burcsi} to
obtain an exhaustive classification of them and to describe their properties. 
The method of coding Hurwitz polynomials
by finite sets of positive rational integers in continued fractions
(Henrici \cite{henrici}, chapter 12) is transported
to expansive polynomials. In the present note, we continue further
this coding towards Salem numbers using Theorem A.

In section \ref{S2} we recall 
the two related subclasses 
$A_q$ and $B_q$
of Salem numbers 
which arise from the interlacing Theorems A and B 
(\cite{bertinboyd}, \cite{bertinpathiauxdelefosse} p. 129 and 133).
Since the set T decomposes as
\begin{equation}
\label{taqbq}
{\rm T} = \bigcup_{q \geq 2} A_{q} = \bigcup_{q \in \nb} B_{q}
\end{equation}
investigating the limit points of T only as limit points of
Salem numbers in the subclasses $A_q$, by the coding by 
continued fractions as presently, and their deformations, 
does not result in a loss of generality.

\vspace{1cm}

\textbf{Notations : }

\begin{definition}
A Salem polynomial is a monic polynomial with integer coefficients 
having exactly one zero (of multiplicity 1) outside the unit circle, and at least one zero 
lying on the unit circle.  
\end{definition}

A Salem polynomial is not necessarily irreducible. If it vanishes at $\theta > 1$, 
and it is reducible, then, by Kronecker's Theorem \cite{kronecker}, 
it is the product of cyclotomic polynomials by the minimal polynomial of $\theta$.
Let $\theta$ be a Salem number. 
The minimal polynomial $T(z)$ of $\theta$ has an even degree $2n$, $n \geq 2$, with simple roots. 
$T(z)$ has exactly one zero $\theta$ of modulus $>1$, 
one zero $\frac{1}{\theta}$ of modulus $<1$ and $2n-2$ zeros on the unit circle, 
as pairs $(\alpha_j, \overline{\alpha_j})$ of complex-conjugates.
The notation `$T$' for Salem polynomials is the same as for the set of Salem
numbers, since it presents no ambiguity in the context.

\begin{definition}
An {\it expansive polynomial} is a polynomial with coefficients in a real subfield
of $\mathbb{C}$, of degree 
$ \geq 1$, such that all its roots in $\mathbb{C}$ 
have a modulus strictly greater than 1.
\end{definition}
An expansive polynomial is not necessarily monic. 

\begin{definition}
Let $P(z)$ be a polynomial $\in \zb[z]$, and  $n=\deg(P)$. 
The {\it reciprocal polynomial} of $P(z)$ is $P^*(z)=z^n P(\frac{1}{z})$.
A polynomial $P$ is a reciprocal polynomial if $P^*(z)=P(z)$.
A polynomial $P$ is an {\it antireciprocal polynomial} if $P^*(z)=-P(z)$ .
\end{definition}

\begin{definition}
If $P(X) = a_0 \prod_{j=1}^{n} (X- \alpha_j)$ is a polynomial
of degree $n \geq 1$ with coefficients in $\cb$, 
and roots $\alpha_j$, the {\it Mahler measure} of $P$ is
$${\rm M}(P) := |a_0| \prod_{j=1}^{n} \max\{1, |\alpha_j|\}. 
$$
\end{definition}

\begin{definition}
A {\it negative Salem number} is an algebraic integer $\theta < -1$ 
such that the Galois conjugates
$\theta^{(i)}$ of $\theta$ satisfy: $|\theta^{(i)}| \leq 1$ with at least one conjugate
of modulus 1.
\end{definition}

In the case where expansive polynomials are irreducible, 
the following definition extends the classical one of Garsia \cite{brunotte0}
\cite{brunotte1} \cite{harepanju}.

\begin{definition}
A {\it generalized Garsia number} is an algebraic integer for which the minimal polynomial
is a monic (irreducible) expansive polynomial with absolute value of the 
constant term greater than or equal to $2$. A generalized Garsia polynomial $P$
is a monic irreducible expansive polynomial with integer coefficients such that
${\rm M}(P) \geq 2$. A Garsia number is a generalized Garsia number of 
Mahler measure equal to $2$. 
\end{definition}

\begin{definition}
The $n$th cyclotomic polynomial, with integer coefficients, 
is denoted by $\Phi_{n}(X), n \geq 1$, with
$\Phi_{1}(X) = X-1, \Phi_{2}(X) = X+1$ and $\deg(\Phi_{n}) = \varphi(n)$ 
even as soon as
$n > 2$. The degree $\varphi(n)=n \prod_{p \, {\rm prime}, \, p|n}(1-1/p)$ of $\Phi_{n}(X)$ is the Euler's 
totient function.
In the sequel, following Boyd \cite{boyd00}, 
we adopt the (non-standard) convention that 
`cyclotomic polynomial' means a monic integer polynomial having
all its roots on the unit circle, i.e. a product of 
$n$th cyclotomic polynomials for various values of $n$.
\end{definition}

\section{Bertin-Boyd Interlacing Theorem A}
\label{S2}

\subsection{The $Q$-construction of a reciprocal (or an anti-reciprocal) polynomial from a polynomial $P$ by an algebraic function}
\label{sec:SS1S2}

Let $\mathbb{K}$ be a subfield of 
$\mathbb{R}$, $P(X) \in \mathbb{K}[X], \deg(P) = n \geq 1$, and
$z$ the complex variable. 
With $ \epsilon = \pm 1 $, the polynomial defined by 
\begin{equation}
\label{defQ}
 Q(z)=z P(z)+\epsilon P^*(z)
\end{equation}
satisfies $Q^*(z)=\epsilon \, Q(z)$. 
It is either a reciprocal (if $\epsilon=+1$), or an anti-reciprocal (if $\epsilon=-1$) polynomial.
The algebrasc function obtained by the
related polynomial :
\begin{equation}
  Q(z,t)=z P(z)+ \epsilon t P^* (z)
\end{equation}
defines
an affine algebraic curve over $\cb$ (first considered by Boyd \cite{boyd0}):
\begin{equation}
\{(z,t) \in \cb^2 \mid Q(z,t) = 0\}.
\end{equation}
For $0 \leq t \leq 1$ 
the equation $Q(z,t)=0$ over $\cb$ defines an algebraic curve $z=Z(t)$ with $n+1$ branches. 
$Z(0)$ is the set of the zeros of $z P(z)$ and $Z(1)$ is the set of the zeros of $Q(z)$.

By $Q$-construction from $P$, over $\mathbb{K}$,
we mean the couple $(Q(z), Q(z,t))$ given by the 
reciprocal or anti-reciprocal polynomial $Q$ and its associated algebraic function
$Q(z,t)$, both having specific properties arising from those of $P$ and the sign of
$\epsilon$, as described below.

On $|z|=1$, let us remark that $|P(z)|=|P^*(z)|$. 
Then $Q(z,t)$ has no zeros on $ |z|=1$ for $0 \leq t < 1$. 
Each branch of the algebraic curve $z=Z(t)$ is 
\begin{itemize}
 \item[(i)] either included in $|z| \leq 1$; when it starts from a zero of $z P(z)$ in $|z|<1$,
 \item[(ii)] or included in $|z| \geq 1$; when it starts from a zero of $z P(z)$ in $|z|>1$. 
\end{itemize}
\noindent
Then a zero of $Q(z)$ on the unit disc $|z|=1$ is 
\begin{itemize}
 \item[(i)] either stemming from a branch included in the unit disc; then it is called an {\em exit},
 \item[(ii)] or stemming from a branch outside the unit disc; then it is called an {\em entrance}.
\end{itemize}

The example of the Lehmer polynomial and the smallest known Salem number, Lehmer's number, is given in Figure \ref{branches}.

\begin{figure}[ht]
 \centering
 \includegraphics[width=11cm]{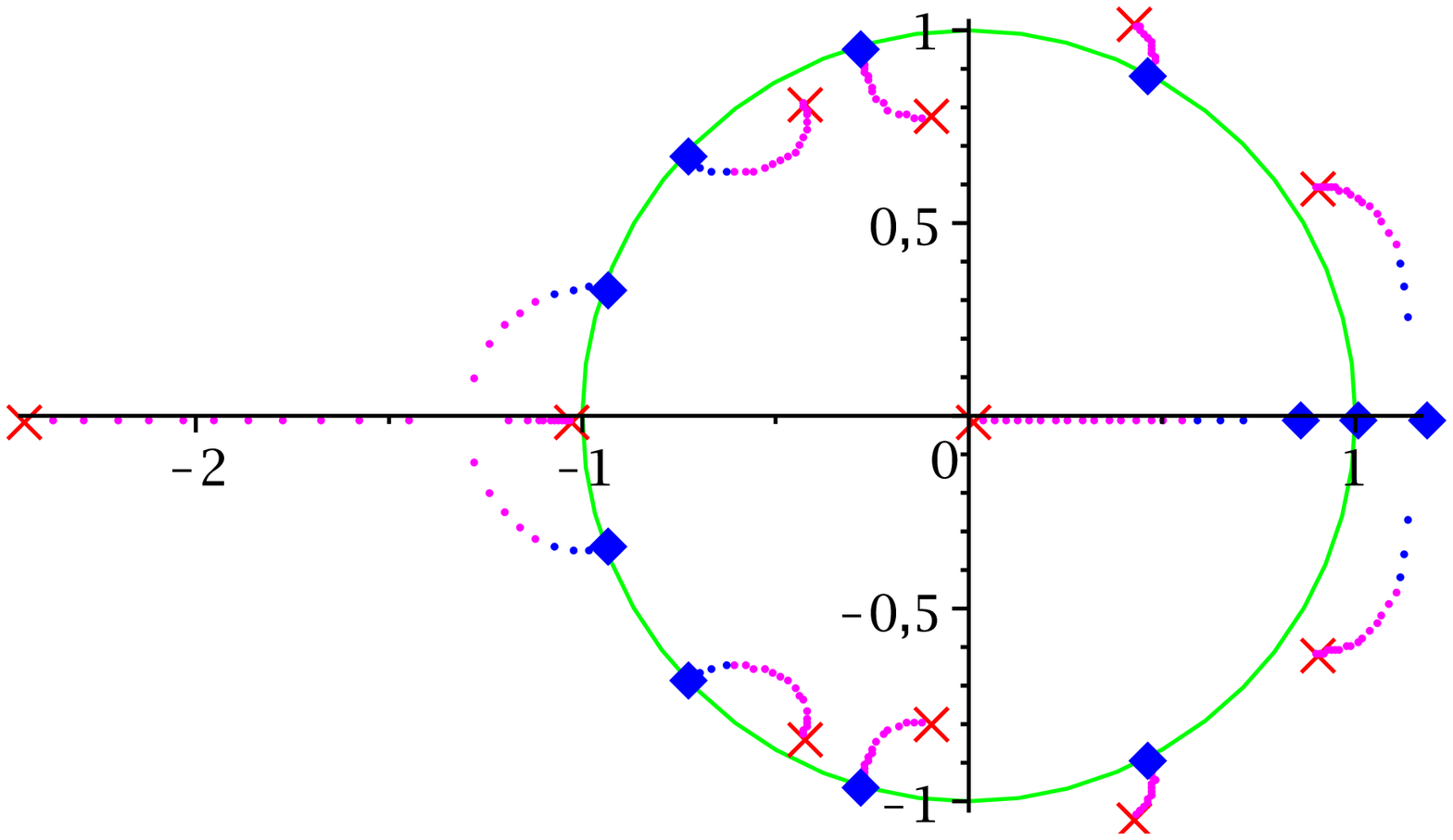}
 \caption{
Branches of the algebraic curve obtained with the monic nonexpansive
polynomial $P(z)=z^{10}+2z^9+3z^7+z^6+2z^5+2z^4+z^3+4z^2+z+2$ (crosses), 
producing by the Q-construction the anti-reciprocal polynomial $Q(z)=(z-1)(z^{10}+z^9-z^7-z^6-z^5-z^4-z^3+z+1)$ (diamonds), 
which has 5 entrances, 4 exits and 1 zero of modulus $>1$, the Lehmer number :
$\theta \approx 1.17628... $.}
\label{branches}
\end{figure}

\subsection{Expansive polynomials of Mahler measure $q$. Classes $A_q$ }
\label{sec:SS2S2}

In the particular case where $\mathbb{K}= \qb$ and $P(X)$ is monic and expansive
with integer coefficients, $z P(z)$ has one zero in the open unit disc 
and $n$ zeros outside the closed unit disc. 
The algebraic curve $z=Z(t)$ has at most one exit and $n$ entrances.
Therefore $Q(z)$ has at most one zero inside $|z|<1$.

Since $Q(z)$ is a reciprocal polynomial, if $Q(z)$ has no zero in $|z|<1$, 
then all his zeros are on the unit circle.
And if $Q(z)$ has exactly one zero $\alpha$ in $|z|<1$, it has exactly 
one zero $\theta$ in $|z|>1$ and $n-2$ zeros on the unit circle. 
Then, if $n \geq 4$, $\theta = \frac{1}{\alpha}$ is a Salem number
and $Q$ is a Salem polynomial.
We say that $\theta$ is produced by the Q-construction from $P$.

\begin{definition}
\label{classAq}
Let $q \in \nb^*$ be a nonzero integer. 
The class ${A_q}$ is the set of Salem numbers produced by the Q-construction
(over $\mathbb{K}=\qb$) from monic expansive 
polynomials $P(X) \in \zb[X]$ having a constant term equal to $\pm q$ and
$\epsilon=-{\rm sgn}P(0)$.
\end{definition}
 
\begin{remark}
\label{qdessus2}
The sets $A_0$ and $A_1$ are empty since all the zeros $\alpha_i$ of any
monic expansive polynomial $P$ are in $|z|>1$, so that we have  $q=| \prod \alpha_i| >1$.
\end{remark}

\begin{definition}
Let $q$ be an integer $\geq 2$. 
The set of monic expansive polynomials $P(z) \in \zb[z]$ 
such that $|P(0)|= {\rm M}(P) = q$ producing a Salem number by the Q-construction 
$ Q(z) = z P(z) + \epsilon P^{*}(z) $ with $\epsilon \in \{ -1 ; +1 \} $, 
is denoted by $E_q$.
\end{definition}

\begin{theorem}[Bertin, Boyd \cite{bertinboyd}, Theorem A]
\label{theoremA}
Suppose that $\theta$ is a Salem number with minimal polynomial $T$. 
Let $q \in  \nb \setminus \{0,1\}$.
Then $\theta$ is in $A_q$ if and only if there is a cyclotomic polynomial $K$ with simple roots 
and $K(1) \neq 0$ and a reciprocal polynomial $L(X) \in \zb[X]$ 
with the following properties :
   \begin{enumerate}[(a)]
    \item $L(0)=q-1$,
    \item $\deg(L) = \deg(KT)-1$,
    \item $L(1) \geq -K(1)T(1)$,
    \item $L$ has all its zeros on $|z|=1$ and they interlace the zeros of $KT$ on $|z|=1$ 
    in the following sense : let $e^{i \psi _1}, ... e^{i \psi _m}$ the zeros of $L$ on 
     $\{{\rm Im} z  \geq 0\} \setminus \{z=-1\}$ with $0 < \psi_1 < ... < \psi_m < \pi$ 
    and let $e^{i \phi _1}, ... e^{i \phi _m}$ the zeros of $KT$ on 
    $\{{\rm Im} z  \geq 0\} $ with $0 < \phi_1 < ... < \phi_m \leq \pi$  , 
    then $0 < \psi_1 < \phi_1 < ... < \psi_m < \phi_m $.
   \end{enumerate}
\end{theorem}
 
\noindent
The construction 
of the polynomials $L$ and $KT$ is explicit 
(\cite{bertinpathiauxdelefosse}, pp 129--133), as follows :
\begin{itemize}
 \item if $\epsilon=-1$ then $Q(1)=0$. Then $P_1$ and $Q_1$ are chosen as $P_1(z)=P(z)$ and $\displaystyle Q_1(z)=\frac{Q(z)}{z-1}$,
 \item else if $\epsilon=+1$, then $P_1$ and $Q_1$ are chosen as $P_1(z)=(z-1)P(z)$ and $Q_1(z)=Q(z)$.
\end{itemize}
In both cases, the polynomial $Q_1(z)$ is a reciprocal polynomial which satisfies the equation:
\begin{equation}
\label{EqDeQ1}
 (z-1)Q_1(z)=zP_1(z)-P_1 ^*(z).
\end{equation}
We say that $P_1$ lies ``over $Q_1$". 
As $Q_1(\theta)=0$, $Q_1$ is the product of the minimal polynomial of $\theta$ by a product of cyclotomic polynomials
as 
\begin{equation}
  K(z)T(z)=Q_1(z),
\end{equation}
and the polynomial $L$, reciprocal by construction, is given by
\begin{equation}
\label{elP1Q1}
 L(z)=P_1(z)-Q_1(z).
\end{equation}

In \S \ref{sec:SS4.1S2}, \S \ref{sec:SS4.2S2} and \S \ref{sec:SS4.3S2} 
we establish existence theorems for the polynomials
$P_1$ associated to a given Salem polynomial $Q_1$ by the equation \eqref{EqDeQ1},  
focusing on the case $\epsilon = -1$, that is with $P(0)=P_1(0) > 0$
and $\deg(P)$ even.
The methods of Geometry of Numbers used call for nondegenerated polyhedral cones
in Euclidean spaces of dimension half the degree of the Salem polynomial. 
These theorems are called {\it association theorems}. In \S \ref{sec:SS4.4S2} 
the second case of association theorem, with $\epsilon = +1$, is
briefly shown to call for similar methods, after a suitable factorization of
\eqref{EqDeQ1} and sign changes.
The algorithmic search
of an expansive polynomial over a Salem polynomial is considered in
\S \ref{sec:SS4.5S2} from a practical viewpoint.

The example of the interlacing of roots of $L$ and $KT$ associated to the Lehmer number is given in Figure \ref{LKT}.

\begin{figure}[ht]
 \centering
 \includegraphics[width=11cm]{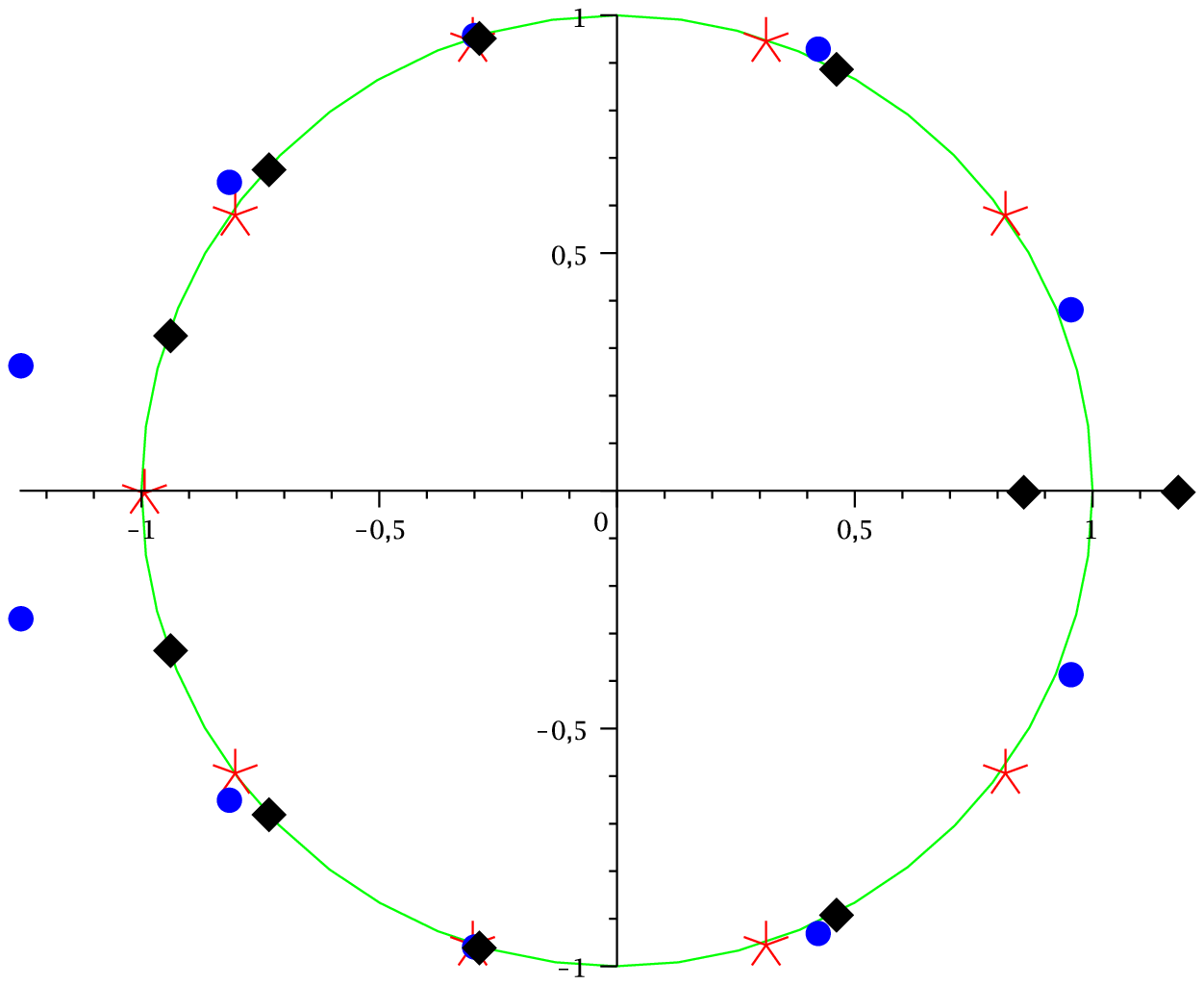}
 \caption{Interlacing of the zeros of $L = P-T$ (asterisks) and those of $KT = T$ (diamonds) on 
the unit circle obtained with the monic expansive polynomial 
 $P(z)=z^{10}+2z^9+z^8-z^7-z^6-z^4-z^3+2z+2$ (circles), 
producing by the Q-construction the anti-reciprocal Salem polynomial 
$Q(z)=(z-1)(z^{10}+z^9-z^7-z^6-z^5-z^4-z^3+z+1) = (z-1)T(z)$  with
the Lehmer number $\theta \approx 1.176\ldots $ as dominant root of $T$.}
 \label{LKT}
\end{figure}

In \S \ref{sec:SS3S2} the type of interlacing provided by Theorem \ref{theoremA}
is revisited in the more general context of interlacing modes
on the unit circle proposed by McKee and Smyth \cite{mckeesmyth}.

The classes $(B_q)$ of Salem numbers
are defined by a similar construction with a polynomial P which has a single zero in $|z|>1$.
We refer to \cite{bertinboyd} (\cite{bertinpathiauxdelefosse}, p. 133) 
for their definition. All the Salem numbers are generated by the classes $A_q$ and $B_q$,
giving rise to \eqref{taqbq}. The distribution of the small
Salem numbers in the classes $A_q$ in intervals was studied by Boyd \cite{boyd0}
and Bertin and Pathiaux-Delefosse \cite{bertinpathiauxdelefosse}. 
Bertin and Boyd \cite{bertinboyd} proved that
for $q \geq 2$ and $k \geq 1$ , $A_2 \subset A_q$ and  $A_q \subset A_{kq-k+1}$.
The distribution in the other classes remains obscure.

\begin{conjecture} [Local Density Conjecture]
\label{ldCCJJ}
For all $c>1$, there exists $M > 0$ such that $T \cap [ 1, c ]$ is contained in a finite union
$\bigcup_{2 \leq q \leq M} A_q$.
\end{conjecture}

\subsection{Interlacing on the unit circle and McKee-Smyth interlacing quotients}
\label{sec:SS3S2}

Following McKee and Smyth \cite{mckeesmyth} \cite{mckeesmyth1}
three types of interlacing conditions, CC, CS and SS, are relevant.

\begin{definition}
 Suppose that $C_1$ and $C_2$ are coprime polynomials with integer coefficients and positive leading coefficients.
We say that $C_1$ and $C_2$ satisfy the {\it \textbf{CC-interlacing condition}} (CC for Cyclotomic-Cyclotomic) 
or  $C_1/C_2$ is a CC-interlacing quotient if
  \begin{itemize}
   \item $C_1$ and $C_2$ have all their roots in the unit circle.
   \item all the roots of $C_1$ and $C_2$ are simple,
   \item the roots of $C_1$ and $C_2$ interlace on $|z|=1$.
  \end{itemize}
\end{definition}

\noindent
{\bf Remarks.}
\begin{itemize}
\item[(i)] As $C_1$ and $C_2$ have the same number of roots, 
$C_1$ and $C_2$ have the same degree; 
\item[(ii)] as the non-real zeros of $C_1$ and $C_2$ are conjugated in complex sense two by two, the reals $-1$ and $+1$ must be in the set of their roots to ensure 
the interlacing on the unit circle; 
\item[(iii)] one polynomial among $C_1$ and $C_2$ is a reciprocal polynomial, the other being an anti-reciprocal polynomial, having $(z-1)$ in its factorization;
\item[(iv)] the terminology $CC$ for ``Cyclotomic-Cyclotomic" is misleading. Indeed, 
$C_1$ and $C_2$ are not necessarily monic, so they are not necessarily cyclotomic 
polynomials.
\end{itemize}

 A complete classification of all pairs of cyclotomic polynomials whose zeros interlace on the unit circle is reported in \cite{mckeesmyth1}.

\begin{definition}
\label{CSentrecroisement} 
Suppose that $C$ and $S$ are coprime polynomials with integer coefficients and positive 
{\it leading} coefficients.
We say that $C$ and $S$ satisfy the {\it \textbf{CS-interlacing condition}} (CS for Cyclotomic-Salem) 
or  $C/S$ is a CS-interlacing quotient if
  \begin{itemize}
   \item $S$ is reciprocal and $C$ is antireciprocal,
   \item $C$ and $S$ have the same degree,
   \item all the roots of $C$ and $S$ are simple, except perhaps at $z=1$,
   \item $z^2-1 \mid C$,
   \item $C$ has all its roots in $|z|=1$,
   \item $S$ has all but two roots in $|z|=1$, with these two being real, positive, $\neq 1$,
   \item the roots of $C$ and $S$ interlace on $\{|z|=1\} \setminus \{1\}$.
  \end{itemize} 
\end{definition}

\begin{definition} 
Suppose that $S_1$ and $S_2$ are coprime polynomials with integer coefficients 
and positive {\it leading} coefficients.
We say that $S_1$ and $S_2$ satisfy the {\it \textbf{SS-interlacing condition}} (SS for Salem-Salem) 
or  $S_2/S_1$ is a SS-interlacing quotient if
  \begin{itemize}
   \item one of $S_1$ and $S_2$ is reciprocal polynomial, the other is an anti-reciprocal polynomial,
   \item $S_1$ and $S_2$ have the same degree,
   \item all the roots of $S_1$ and $S_2$ are simple, 
   \item $S_1$ and $S_2$ have all but two roots in $|z|=1$, with these two being real, positive, $\neq 1$,
   \item the roots of $S_1$ and $S_2$ interlace on $\{|z|=1\} \setminus \{1\}$.
  \end{itemize} 
\end{definition}

\begin{remark}
There are two types of SS-interlacing condition : if the largest real roots of $ S_1 S_2 $ is a root of $S_1$, 
then $S_2 / S_1$ is called a 1-SS-interlacing quotient, and $S_1/S_2$ is called a 2-SS-interlacing quotient.
\end{remark}

Theorem \ref{theoremA} provides interlacing on the unit circle; more precisely,
referring to \eqref{EqDeQ1} and denoting $n :=\deg(Q_1)=\deg(P_1)$, 
let us show that if $n$ is
\begin{itemize}
\item[(i)] even, the quotient $(z-1)L/KT$ is a CS-interlacing quotient, 
\item[(ii)] odd, no CS-interlacing condition is satisfied.
\end{itemize}

Indeed, if $n$ is even, from \eqref{elP1Q1}, with $L=P_1 - Q_1$,
$$L(-1) = P_1(-1)-\frac{-P_1(-1)-(-1)^n P_1(\frac{1}{-1})}{(-1-1)} 
= P_1(-1)-\frac{-2P(-1)}{-2} = 0.$$
Then the factor $(z+1)$ divides $L$ and we can take $C= (z-1)L $ and $S = KT$.   
As $L$ and $KT$ are reciprocal polynomials, $(z-1)L$ is an anti-reciprocal polynomial.
Moreover, $\deg(L)=\deg(KT)-1$, so $(z-1)L$ and $KT$ have the same degree.
Finally, by definition, $(z^2-1) | C$.
The roots of $C$ and $S$ are simple and all but two roots of 
$S$ interlace the roots of $C$  on $\{ |z|=1 \} \setminus \{ z=1 \} $, 
S having two real roots being positive inverse and $\neq 1$.

On the contrary, if $n$ is odd, 
$$Q_1(-1) = \frac{-P_1(-1)-(-1)^n P_1(\frac{1}{-1})}{(-1-1)} = 0.$$ 
Then the factor $(z+1)$ divides the Salem polynomial $Q_1=KT$;
-1 cannot be a zero of L.
The item $(z^2 - 1) \mid C$  in the CS-interlacing condition is then missing. 

Let us turn to the Salem numbers associated with by these interlacing modes.
McKee and Smyth \cite{mckeesmyth} use the variant $z \mapsto x=\sqrt{z}+1/\sqrt{z}$
of the Tchebyshev transformation, instead of the more usual one
$z  \mapsto x = z + 1/z$, to study the Salem numbers produced by the above different cases
of interlacing quotients. 

\begin{theorem}[\cite{mckeesmyth}, Theorem 3.1]
Let $C_2 /C_1$ be a CC-interlacing quotient with $C_1$ monic, of degree $\geq 4$. 
By the map $x=\sqrt{z} + \frac{1}{\sqrt{z}}$
the function $ \displaystyle \frac{\sqrt{z} C_2(z) }{(z-1) C_1(z)}$ is transformed 
into the real interlacing quotient $\displaystyle \frac{c_2(x)}{c_1(x)}$ 
with $c_1$ and $c_2$ coprime polynomials in $\zb \left[ x \right] $. 
If  $\displaystyle \lim _{x -> 2 ^+} \frac{c_2(x)}{c_1(x)} > 2$ then the solutions of the equation 
\begin{equation}
\label{NewQuotient}
 \frac{C_2(z)}{(z-1) C_1(z)}=1 + \frac{1}{z}
\end{equation}
are a Salem number, its conjugates and roots of unity.
\end{theorem}

\begin{theorem}[\cite{mckeesmyth}, Theorem 5.1]
\label{mcsm51}
 Let $C/S$ be a CS-interlacing quotient with $S$ monic, of degree $\geq 4$.
 The solutions of the equation 
 \begin{equation}
 \frac{C(z)}{(z-1) S(z)}=1 + \frac{1}{z}
\end{equation}
are a Salem number, its conjugates and roots of unity.
\end{theorem}

Theorem \ref{theoremA} provides a CS-interlacing quotient if 
the common degree $n = \deg(Q_1) = \deg(P_1)$ is even, as mentioned above;
then the quotient $(z-1)L/KT$ is a CS-interlacing quotient, with $KT$ a monic polynomial. 
As such, we can now apply Theorem \ref{mcsm51} to this quotient. 
 This theorem offers the construction of another Salem polynomial $T_2$ : 
 \begin{equation}
  T_2(z)=zL(z)-(z+1)K(z)T(z).
 \end{equation}
 Denote $\theta$ the dominant root of $T$ and $\theta_2$ the dominant root of $T_2$.
Remark that $\theta_2$ is always different from $\theta$. 
Otherwise, if $\theta_2=\theta$ then 
$\theta L(\theta)= T_2(\theta)+(z+1)K(\theta)T(\theta)=0$, which is impossible because all zeros of $L$ lies on the unit circle.
 
For exemple, with the smallest Salem known number $\theta \approx 1.1762808$ of degree 10, 
root of $T(z)=z^{10}+z^9-z^7-z^6-z^5-z^4-z^3+z+1$, 
we obtain by this construction, with the expansive polynomial $P(z)= x^{10} + 2 x^9 + x^8 + x^2 + 2 x + 2$ over $T$,  
the Salem number $\theta_2 \approx 1.5823471$ of degree 6, 
root of the Salem polynomial $(z+1)(z^2+z+1)(z^2-z+1)(z^6-z^4-2z^3-z^2+1)$. 
We remark that this sequence depends on the choice of the polynomial $P$ over $T$.

\begin{theorem}[\cite{mckeesmyth}, Theorem 5.2]
 Let $S_2/S_1$ be an SS-interlacing quotient with $S_1$ monic.
 If $\displaystyle \lim_{z-> 1+} \frac{S_2(z)}{(z-1)S_1(z)} <2$
 then the solutions of the equation
 \begin{equation}
 \frac{S_2(z)}{(z-1) S_1(z)}=1 + \frac{1}{z}
\end{equation}
are a Salem number, its conjugates and roots of unity.
\end{theorem}

Under the assumption that Lehmer's Conjecture is true,
Theorem 9.2 in \cite{mckeesmyth} shows that
the smallest Salem number $\theta$ is such that there exists
a type 2 SS-interlacing quotient
$S_1/S_2$, with two monic polynomials $S_1$ and $S_2$ 
satisfying
\begin{equation}
\label{typeIV}
\frac{S_2(z)}{(z-1) S_1(z)} = \frac{2}{1+z}
\end{equation}
for which the only solutions of \eqref{typeIV}
are $z=\theta$, its conjugates and perhaps some
roots of unity. The other small Salem numbers 
are probably produced by type 2 SS-interlacing quotients with a
condition of type \eqref{typeIV} as well; comparing with 
the Local Density Conjecture
in section 2.2, they are obtained from expansive polynomials 
of small Mahler measure
equal or close to $2$.

{\it Extension of the field of coefficients:} CC-interlacing quotients
were extended by Lakatos and Losonczi \cite{lakatoslosonczi} 
to classes of reciprocal polynomials  
having coefficients in $\mathbb{R}$.

\subsection{Association Theorems between expansive polynomials and Salem polynomials}
\label{sec:SS4S2}

The classes $A_q$ of Salem numbers, in Theorem \ref{theoremA}, 
call for two disjoint classes
of monic expansive polynomials $P$: those for which 
the constant term $P(0)$ is positive (case $\epsilon = -1$), 
those for which it is negative (case $\epsilon = +1$).
Below we focus on the existence of expansive polynomials over
a Salem polynomial when the Mahler measure
M$(P)$ is equal to $P(0)$ (case $\epsilon = -1$). 
In \S \ref{sec:SS4.4S2} we indicate how the 
previous construction can be adapted to deduce existence theorems
in the second case $\epsilon = +1$.

\subsubsection{A criterion of expansivity}
\label{sec:SS4.1S2}

\begin{theorem}
\label{QtoP}
Let $T$ be an irreducible Salem polynomial of degree $m \geq 4$. 
Denote by $\beta>1$ its dominant root.  
Let $P$ be a polynomial $\in \rb \left[ z \right] $ of degree $m$ such that 
 \begin{equation}
 \label{PtoT}
   (z-1)T(z) =  z P(z) - P^*(z). 
 \end{equation}
\noindent
$P$ is an expansive polynomial if and only if 
 \begin{itemize}
  \item  $P(1) T(1) < 0$, and,
  \item for every zero $\alpha$ of $T$ of modulus $1$, 
   $$(\alpha-1)\alpha^{1-m}P(\alpha)T'(\alpha)~~ \mbox{is real and negative}.$$
 \end{itemize}
 
\end{theorem}

\begin{remark}
The equation \eqref{PtoT} implies that P is monic and 
 denoting $P(z)= p_m z^m + p_{m-1} z^{m-1} + .. + p_1 z + p_0$, then  $p_i \in \zb$ with  
$|p_0| > p_m = 1$.  
Since the RHS of \eqref{PtoT} vanishes at $z=1$,
the factorization of the LHS of \eqref{PtoT} by $z-1$ is natural.
Denote $Q(z):=(z-1)T(z)$.
\end{remark}
 
\begin{proof}
{\it The conditions are necessary:}
as $T$ is reciprocal, $Q$ is anti-reciprocal of degree $m+1$;  
$Q$ has exactly one zero $\beta$ outside the closed unit disc $\overline{D(0;1)}$, 
 one zero $\frac{1}{\beta}$ in $D(0;1)$, and $m-1$ zeros on $C(0;1)$, 
  which are $z=1$ and $ \displaystyle \frac{m}{2}-1$ 
pairs of complex conjugates $ ( \alpha_i, \overline{\alpha_i}) $.
Let $Q_t$ be the parametric polynomial 
\begin{equation}
\label{def_de_Qt}
 Q_t (z)=z P(z)-t P^* (z).
\end{equation}
On the unit circle $C(0;1)$, 
we have $|z P(z)| = |z| . |P(z) | = |P(z)| = |P^* (z)|$, and $|P(z)| \neq 0$ 
because P is an expansive polynomial.
Then, for $0 < t < 1$,  $|Q_t(z) - z P(z)|=|- t P^*(z)|  < |P^*(z)|=|z P(z)|$.
The theorem of Rouch\'e implies that the polynomials $Q_t(z)$ and $zP(z)$ have 
the same number of zeros in the compact $\overline{D(0;1)}$.
So $Q_t(z)$ has exactly one zero in $\overline{D(0;1)}$ and $m$ zeros 
outside.

By \eqref{def_de_Qt}, for $t > 1$, $\displaystyle Q_{\frac{1}{t}}(\frac{1}{z})= \frac{-1}{t z^{m+1}} Q_t(z)$.
Then, if $\alpha $ is a zero of $Q_t$, then $\displaystyle \frac{1}{\alpha}$ is a zero of  $ \displaystyle Q_{\frac{1}{t}}(\frac{1}{z})$. 
Moreover, $ \displaystyle \frac{1}{\overline{\alpha}}$ is a zero of  
$ \displaystyle Q_{\frac{1}{t}}(\frac{1}{z})$ as well because $\displaystyle Q_{\frac{1}{t}} \in \rb \left( z \right) $
(thus we obtain the zeros of $\displaystyle Q_{\frac{1}{t}} $ from those of $Q_t$ by an inversion of centre $0$ of radius $1$).
Let $\displaystyle f(z):=\frac{Q(z)}{P^*(z)}$.
Then, by \eqref{def_de_Qt},
\begin{equation}
\label{relQtf}
 \frac{Q_t(z)}{P^*(z)} =  f(z) + (1-t).
\end{equation}
For $\displaystyle \alpha \in \{ \frac{1}{\beta}, \beta, 1 , \alpha_1, \overline{\alpha_1},\alpha_2, \overline{\alpha_2}, ..., 
\alpha_{\frac{m}{2}-1}, \overline{\alpha_{\frac{m}{2}-1}   } \}$, when $z$ lies 
in a neighbourhood
of $\alpha$, the equation $Q_t(z)=0$ is equivalent to the equation $f(z)=t-1$.
As $Q$ has simple zeros, $Q(\alpha)=0$ and $Q'(\alpha) \neq 0$. 
Then $f(\alpha)=0$ and 
$f'(\alpha)=\frac{Q'(\alpha)P^*(\alpha)-Q(\alpha)(P^*)'  (\alpha)}{P^*(\alpha)^2}=\frac{Q'(\alpha)}{P^*(\alpha)} \neq 0$.
By the local inversion theorem, in the neighbourhood of $t=1$, 
there exist an analytic function 
$h_{\alpha}$ such that the equation $Q_t(z)=0$ is equivalent to
$z=h_{\alpha}(t)$, with $h_{\alpha}(1)=\alpha$ and $h'_{\alpha}(1) \neq 0$.
Then $\{ h_{\alpha}(t) ;  \alpha \in \{ \frac{1}{\beta}, \beta, 1 , (\alpha_i )_{ 1 \leq i <m } \} \}$ is the set of zeros of $Q_t(z)$.
In the neighbourhood of $t=1$, we have : 
\begin{equation}
\label{taylor}
 h_{\alpha}(t)=h_{\alpha}(1) + (t-1) h'_{\alpha}(1)+...
\end{equation}
By the inversion property of $Q_t$, if $\displaystyle h_{\alpha}(t)$ is 
a zero of $\displaystyle Q_t(z)$ then $\displaystyle 1 / \overline{h_{\alpha}(t)}$ 
is a zero of $\displaystyle Q_{\frac{1}{t}}$ : 
there exist $\displaystyle \widetilde{\alpha}  \in \{ \beta^{-1}, \beta, 1 , (\alpha_i )_{ 1 \leq i <m } \} $ such that 
\begin{equation}
\label{InvProp}
 \frac{1}{\overline{h_{\alpha}(t)}}=h_{\widetilde{\alpha}}(\frac{1}{t}).
\end{equation}
When $t=1$, we obtain $\displaystyle 1 / \overline{h_{\alpha}(1)}= 1 / \overline{\alpha}
=\alpha / |\alpha|^2$ and 
$\displaystyle h_{\widetilde{\alpha}}(1)=\widetilde{\alpha}$.
In particular, for any 
$\alpha$ of modulus 1,
we obtain
$\displaystyle \widetilde{\alpha}=\alpha$ and 
 $\displaystyle 1 / \overline{h_{\alpha}(t)}=h_{\alpha}(\frac{1}{t})$, that is
  \begin{equation}
  \label{eqdeh}
    {h_{\alpha}(t)} \overline{h_{\alpha}(\frac{1}{t})}=1.
  \end{equation}
In this case, we denote $\displaystyle h_{\alpha}(t)=X(t)+iY(t)$. Then \eqref{eqdeh} becomes
 $\displaystyle (X(t)+iY(t))(X(1/t)-iY(1/t))=1$.
 The imaginary part of this equation is $\displaystyle Y(t)X(1/t) - X(t) Y(1/t)=0$.
On differentiation we obtain for $t=1$: $\displaystyle Y'(1)X(1)-Y(1)X'(1)=0$.
Thus $\displaystyle X'(1) / X(1) = Y'(1) / Y(1)$. 
Let $\displaystyle \lambda \in \rb$ be this quotient.  
Thus $\displaystyle h'_{\alpha}(t)= X'(t)+iY'(t)=\lambda (X(t)+iY(t))=\lambda h_\alpha(t)$, with 
$\displaystyle h'_{\alpha}(1)=\lambda \alpha$ for $t = 1$.
For any $\alpha$ on the unit circle, equation \eqref{taylor} gives  
 \begin{equation}
\label{taylor1}
  h_{\alpha}(t)=\alpha \left[ 1 + (t-1) \lambda +...\right].
 \end{equation}
Since $\displaystyle \alpha \neq \frac{1}{\beta}$, then 
$\displaystyle |h_{\alpha}(t)| > 1=|\alpha|$ for $0<t<1$, implying $\lambda < 0$.   
As $h_{\alpha}$ satisfies the equation \eqref{relQtf}, we have : 
 \begin{equation}
\label{best}
   Q(h_{\alpha}(t))+(1-t)P^*(h_{\alpha}(t))=0.
 \end{equation}
Deriving \eqref{best} at $t=1$, we obtain: 
$\displaystyle h'_{\alpha}(1) Q'(\alpha) - P^* (h_{\alpha}(1))=0$. 
We deduce $\displaystyle h'_{\alpha}(1)= P^*(\alpha) / Q'(\alpha)$.
 Then, 
$$
\begin{aligned}
 0 > \lambda &= \frac{h'_{\alpha}(1)}{\alpha} = \frac{P^*(\alpha)}{\alpha Q'(\alpha)} = \frac{{\alpha}^m P(\frac{1}{\alpha})}{\alpha Q'(\alpha)} \\
  &=\frac{{\alpha}^{m-1} P(\overline{\alpha}) P(\alpha)}{P(\alpha)Q'(\alpha)} = \frac{\left|P(\alpha)\right|^2 }{{\alpha}^{1-m} P(\alpha)Q'(\alpha)}.
\end{aligned}
$$

We deduce $\displaystyle 
  \alpha^{1-m} P(\alpha)Q'(\alpha) <0$, for $\alpha = 1$ or any root of $T$
of modulus 1. Let us transform these inequalities as a function of $T$.
For $\alpha = 1$, $Q'(1) = T(1)$, 
since $\displaystyle Q'(z)=T(z)+(z-1)T'(z)$, and we readily obtain : 
 \begin{equation}
  P(1)T(1) <0.
 \end{equation}
And if $\alpha \neq 1$ is a root of $T$ of modulus $1$, 
since $\displaystyle Q'(\alpha)=(\alpha-1)T'(\alpha)$,
 \begin{equation}
 \label{Entree}
  (\alpha - 1) \alpha^{1-m} P(\alpha)T'(\alpha) <0.
 \end{equation}
We remark that  $(\alpha - 1) \alpha^{1-m} P(\alpha)T'(\alpha) <0 \Leftrightarrow  (\overline{\alpha} - 1) \overline{\alpha}^{1-m} P(\overline{\alpha})T'(\overline{\alpha}) <0$
since both quantities are real: the condition \eqref{Entree} is related to the pair 
$(\alpha, \overline{\alpha})$ for any $\alpha$ of modulus 1. Hence the claim.

{\it The conditions are sufficient:}
first let us show that $P$ has no zero of modulus $1$.
Suppose the contrary: that there exist $\alpha, |\alpha| = 1$,
such that $P(\alpha)=0$.
 Then, as $P$ is in $\displaystyle \rb \left[ z \right]$, $\displaystyle \overline{\alpha}$ is a zero of $P$.
 Then, $\displaystyle Q(\alpha)=\alpha P(\alpha) - (\alpha)^m P(\overline{\alpha})=0$. 
So $\alpha$ would be a zero of modulus $1$ of $Q(z)=(z-1)T(z)$. The only possibilities are
$z=1$ and the zeros of $T$ of modulus 1.
 If $\alpha =1$, then the condition $0=P(1)T(1)<0$ leads to a contradiction.
Similarly, if $\alpha \neq 1$, 
then $\displaystyle 0=(\alpha-1)\alpha^{1-m}P(\alpha)T'(\alpha)<0$ would also be impossible.

Let us show that $zP(z)$ has one zero in $D(0;1)$ (which is $z=0$) 
and $m-1$ zeros outside  $\overline{D(0;1)}$.
Let $\alpha $ be a zero of $T$ of modulus $1$. Let $h_{\alpha}$ defined as in
\eqref{taylor} with \eqref{taylor1}: 
 $$\displaystyle h_{\alpha}(t)= \alpha + h'_{\alpha}(1)(t-1) + ... $$ 
The assumption $\displaystyle (\alpha-1)\alpha^{1-m}P(\alpha)T'(\alpha)<0$ (or  $  P(1) T(1) < 0$ if $\alpha = 1$) 
implies that $\displaystyle |h_{\alpha}(t)| >  |\alpha|=1$ for $t$  in the neighbourhood of $1$,  $t < 1$. 
Thus, in this neighbourhood, $Q_t$ has at least $m-2$ zeros outside $\displaystyle \overline{D(0;1)}$. 
 These zeros belong to the algebraic branches which end at the zeros $\alpha$ of $Q$ of modulus 1.
 
Moreover, as $|P(z)|=|P^*(z)|$, for $|z|=1$, we have : $|zP(z)|=|P(z)| = |P^*(z)| > |tP^*(z)|$
for all $t$, $0 < t < 1$. Hence $Q_t$ has no zero on the unit circle, for 
all $t$, $0<t<1$. By continuity of the algebraic curves defined by $Q_t(z)=0$, the branch 
ending at $\beta$ is included in $\cb \setminus \overline{D(0;1)}$ : 
 this branch originates from a root of $zP(z)$ which lies outside $\overline{D(0;1)}$.
 In the same way, the branch ending at $\displaystyle 1 / \beta$ 
originates from a root of $zP(z)$ which is inside $D(0;1)$.
Therefore $P$ is expansive.
\end{proof}

The above Criterion of expansivity, i.e. Theorem \ref{QtoP}, only 
involves conditions at the roots of $Q$ of modulus 1. However, 
though the existence of expansive polynomials $P$ satisfying
\eqref{PtoT} is only proved below in \S \ref{sec:SS4.2S2}, 
the following Proposition shows that two extra inequalities
at the Salem number $\beta$ and its inverse $\beta^{-1}$ 
should also be satisfied.
 
\begin{proposition}
\label{QtoPplusplus}
 Let $T$ be an irreducible Salem polynomial of degree $m \geq 4$. 
Denote by $\beta>1$ its dominant root.  
 Let $P$ be an expansive polynomial $\in \rb \left[ z \right] $, of degree $m$, such that 
 \begin{equation}
   (z-1)T(z) =  z P(z) - P^*(z). 
 \end{equation}
Then, the polynomial $P$ satisfies the two properties:
$$ {\rm (i)} ~P(1 / \beta) T'(1 / \beta) < 0 \qquad {\rm and} \qquad
{\rm (ii)} ~P(\beta) T'(\beta) > 0.$$
\end{proposition}

\begin{proof}
Let $Q_t$ and $h_{\alpha}$ defined as above in \eqref{def_de_Qt} 
and \eqref{taylor} respectively.
For $\displaystyle \alpha \in \{ \frac{1}{\beta},\beta\}$, the relation \eqref{InvProp} gives : 
 $\displaystyle \widetilde{\alpha}=\frac{1}{\alpha}$ 
and then $1/\overline{h_{\alpha}(t)} = h_{\frac{1}{\alpha}}(\frac{1}{t})$.
 
(i) {\it Case $\alpha=1/\beta$} : for $0 \leq t \leq 1, \,\deg(Q_t(z))=m+1 \geq 5$ is odd. 
 As $Q_t(z) \in \rb \left[ z \right], \, Q_t$ has at least one real root  
and pairs of complex-conjugated roots. 
Since it admits only one root in $D(0;1)$ this zero is real for symmetry reasons; 
and the branch starting at $z=0$ and ending at $\displaystyle z=1/\beta$ is included in $\rb$ (i.e :  
 $\displaystyle h_{\frac{1}{\beta}}(t) \in \rb$ for $0<t<1$). 
 Then, $\displaystyle h'_{\frac{1}{\beta}}(1)=
P^*(\frac{1}{\beta}) / Q'(\frac{1}{\beta}) > 0 $ implying 
 $ \displaystyle P(\frac{1}{\beta}) Q'(\frac{1}{\beta}) > 0 $.
Since $\displaystyle Q'(\frac{1}{\beta})=(\frac{1}{\beta} - 1) T'(\frac{1}{\beta})$  
we readily obtain: 
 $P(\frac{1}{\beta})T'(\frac{1}{\beta}) <0$.

(ii) {\it Case $\alpha=\beta$:} 
we have $\displaystyle h_{\beta}(t)=1/\overline{h_{\frac{1}{\beta}}(\frac{1}{t})}
=1/h_{\frac{1}{\beta}}(\frac{1}{t})$. On differentiation we have: 
$\displaystyle h'_{\beta}(t)= - \frac{-1}{t^2} h'_{\frac{1}{\beta}}(\frac{1}{t})/(h_{\frac{1}{\beta}}(\frac{1}{t}))^2 $;
therefore, for $t=1$ :  
 $\displaystyle h'_{\beta}(1)= h'_{\frac{1}{\beta}}(1)/(h_{\frac{1}{\beta}}(1))^2 
=\beta^2 \frac{P^*(\frac{1}{\beta})}{Q'(\frac{1}{\beta})} = \beta^{2} h'_{\frac{1}{\beta}}(1).$
 The two nonzero real numbers 
$\displaystyle h'_{\beta} (1)$ and $h'_{\frac{1}{\beta}}(1)$
are simultaneously positive or negative.
They are positive. 
As $\displaystyle h'_{\beta} (1) = P^*(\beta)/Q'(\beta) = 
\beta P(\beta)/(\beta-1)T'(\beta) = \frac{\beta}{\beta - 1} 
\frac{P(\beta)^2}{ P(\beta) T'(\beta)}$, we deduce: 
$  P(\beta) T'(\beta) > 0.$
\end{proof}

\subsubsection{Proof of Theorem \ref{main1} for an irreducible Salem polynomial}
\label{sec:SS4.2S2}

Let $T$ be an irreducible Salem polynomial of degree $m \geq 4$. 
Denote $\displaystyle T(z):=z^m+t_{1} z^{m-1} + ...t_{\frac{m}{2}-1} z^{\frac{m}{2}+1} + 
t_{\frac{m}{2}} z^{\frac{m}{2}} + t_{\frac{m}{2}-1} z^{\frac{m}{2}-1} ... + t_1 z + t_0$ 
and $\displaystyle P(z):=z^m+p_{m-1} z^{m-1} + ... + p_1 z + p_0$. 
Though $P$ is nonreciprocal, only half of the coefficient
vector $(p_i)_{i=0,\ldots,m-1}$ entirely determines $P$:
indeed, using the fact that $T$ is reciprocal, the polynomial identity \eqref{main1eq}
gives the following $m/2$ relations between the coefficients:
\begin{equation}
\label{lienPiTi}
 p_{m-i} = p_{i-1} + t_{i}-t_{i-1}, \qquad 1 \leq i \leq \frac{m}{2}.
\end{equation}
The problem of the existence of $P$ is then reduced to finding
a $m/2$-tuple of integers $(p_i)_{i=0,\ldots,m/2 - 1} \in \zb^{m/2}$, characterizing
the ``point" $P$ in the lattice $\zb^{m/2}$, satisfying
the $m/2$ conditions of the Criterion of expansivity of Theorem \ref{QtoP}.
In terms of the coefficient vectors these conditions are linear, each of them
determining an affine hyperplane in $\rb^{m/2}$. The Criterion of expansivity
means that the ``point" $P$ should lie in the intersection of
the $m/2$ open half-spaces defined by these hyperplanes. Let us call
this intersection {\it admissible cone}. We will show that this facetted
polyhedral cone is nonempty; hence it will contain infinitely many
points of the lattice $\zb^{m/2}$. In other terms the number of monic
expansive polynomials $P$ lying over $T$ will be shown to be infinite.

Let us make explicit the equations of the delimiting hyperplanes of the cone,
from Theorem \ref{QtoP}, using \eqref{lienPiTi}.

\vspace{0.1cm}

\noindent
{\it The half-space given by $P(1)T(1) < 0$:}
the condition at $z=1$ is    
\begin{equation}
\label{P1T1facette}
P(1)T(1)=\sum_{i=0}^m p_i \sum_{i=0}^m t_i
=\sum_{i=0}^{\frac{m}{2}-1}  2(2\sum_{k=0}^{\frac{m}{2}-1} t_{k}  +t_{\frac{m}{2}})  p_{i}
+ t_{\frac{m}{2}} (2 \sum_{k=0}^{\frac{m}{2}-1} t_{k}  +t_{\frac{m}{2}}) < 0.
\end{equation}
Combining terms this inequality can be written: 
$\lambda_0 +  \sum_{i=0}^{\frac{m}{2}-1}  a_{0,i} p_i < 0$ with 
$a_{0,0}, a_{0,1}, \ldots, a_{0,\frac{m}{2}-1}, \lambda_0 \in \zb $ only
functions of the coefficients $t_i$ of $T$.

\vspace{0.1cm}

\noindent
{\it The half-space given by $(\alpha-1)\alpha^{1-m}P(\alpha)T'(\alpha)<0$:}
the condition at $z=\alpha \neq 1, |\alpha|=1$, a zero of $T$, is
$ (\alpha-1) \alpha^{1-m} (\sum_{i=0} ^{m} p_i \alpha^i) (\sum_{i=1}^{m} i t_i \alpha^{i-1} )<0$, i.e.
  $$(\alpha - 1)(\alpha^{\frac{m}{2}} + \alpha^{\frac{m}{2}-1} + t_{\frac{m}{2}})(\frac{m}{2}t_{\frac{m}{2}} 
 + \sum_{k=0}^{\frac{m}{2}-1}(k \alpha^{k-\frac{m}{2}} + (m-k)\alpha^{\frac{m}{2}-k}))$$
 \begin{equation}
\label{ALPHAfacette}
+ \sum_{i=0}^{\frac{m}{2}-1} (\alpha^{i-\frac{m}{2}}+\alpha^{\frac{m}{2}+1-i})(\alpha-1)(\frac{m}{2}t_{\frac{m}{2}}+ 
 \sum_{k=0}^{\frac{m}{2}-1} t_k (k\alpha^{k-\frac{m}{2}} + (m-k)\alpha^{\frac{m}{2}-k}))p_i < 0.
\end{equation}
Combining terms this inequality can be written: 
 $\lambda_{\alpha} +  \sum_{i=0}^{\frac{m}{2}-1}  a_{\alpha,i} p_i < 0$ 
with $a_{\alpha,0}, a_{\alpha,1}, \ldots, a_{\alpha,\frac{m}{2}-1}, \lambda_{\alpha}$ 
real in the algebraic number field which is the splitting field
of the polynomial $T$.
The set of solutions of the following linear system of inequalities 
in $\displaystyle \rb^{\frac{m}{2}}$: 
 \[
\label{affineFACETTAGE}
 \begin{cases}  
  \lambda_0 +   a_{0,0} p_0 +a_{0,1} p_1 +...a_{0,{\frac{m}{2}-1}} p_{\frac{m}{2}-1}   & < 0, \\
   \lambda_1 +   a_{1,0} p_0 +a_{1,1} p_1 +...a_{1,{\frac{m}{2}-1}} p_{\frac{m}{2}-1}   & < 0, \\
   ... \\
    \lambda_{\frac{m}{2}-1} +   a_{{\frac{m}{2}-1},0} p_0 +a_{{\frac{m}{2}-1},1} p_1 +...a_{{\frac{m}{2}-1},{\frac{m}{2}}-1} p_{\frac{m}{2}-1}   & < 0, 
 \end{cases}
 \] 
is the admissible cone $\emph{C}^{+}$. 
Let us show that it is nonempty.  
We have just to show that the face hyperplanes $\emph{H}_i$ defined by the equations 
 $\lambda_i +   a_{i,0} p_0 +a_{i,1} p_1 +...a_{i,{\frac{m}{2}-1}} p_{\frac{m}{2}-1}   = 0$ 
intersect at a unique point.
 The linear system of equations corresponding to $\bigcap_{i=0}^{\frac{m}{2}-1} \emph{H}_i$ 
is equivalent to 
  \[
 \begin{cases}  
  P(1)T(1)&=0,\\
  (\alpha_1-1) \alpha_1^{1-m} P(\alpha_1) T'( \alpha_1 )&=0, \\
   ... \\
    (\alpha_{\frac{m}{2}-1}-1) \alpha_{\frac{m}{2}-1}^{1-m} P(\alpha_{\frac{m}{2}-1}) T'( \alpha_{\frac{m}{2}-1} )&=0, 
 \end{cases}
 \]
  where $\alpha_i$ are the simple zeros of $T$ of modulus $1$.
Obviously this system is equivalent to
$   P(1) = P(\alpha_1) = \ldots = P(\alpha_{\frac{m}{2}-1}) =0$,
since $T$, being irreducible, 
has simple roots and that $T(1) \neq 0$. 
The only monic polynomial $P$ in $\rb\left[ z \right]$, of degree $m$, which vanishes
at  
$\displaystyle  \{1, \alpha_1, ..., \alpha_{\frac{m}{2}-1} \}$
is of the form $P_0(z)=(z-1)(z-r) \prod_{i=1}^{\frac{m}{2}-1} (z^2-(\alpha_i+\overline{\alpha_i})z+1)$. 
From \eqref{PtoT},
$P_0$ satisfies $P_0(\beta)= \beta^{m-1} P_0(1/\beta)$, thus $r=\frac{1}{2}(\beta + 1 / \beta)$.
  The half-coefficient vector 
$\displaystyle \mathcal{M}_0=\mbox{}^{t}(p_i)_{0 \leq i \leq  {\frac{m}{2}-1}}$ of $P_0$  
is the unique solution of this system of equations 
(the notation $\mathcal{M}_t$ is that of
\S \ref{sec:SS4.5S2}). 
Remark that $p_0=\frac{1}{2}(\beta+1 / \beta)$.
The point $\mathcal{M}_0$ is the summit vertex of the admissible cone $\emph{C}^{+}$. 
Since the admissible cone is nonempty, it
contains infinitely many points of $\zb^{m/2}$. Hence the claim.

 \begin{proposition}
\label{EqFINITE}
 Let $q \geq 2$ be an integer and $T$ an irreducible 
 Salem polynomial of degree $m$. 
  The subset of $E_q$ of the monic expansive polynomials $P$ over $T$, 
  of degree $m$, such that $P(0)=q$, is finite.
 \end{proposition}

\begin{proof}
In \cite{burcsi}, Burcsi proved that the number of expansive polynomials of the form $z^m + p_{m-1} z^{m-1} + ... + p_1 z + p_0 \in \zb[z]$ with $|p_0|=q$  is finite 
and it is at most $\prod_{k=1} ^{m-1} (2 B( p_0, m, k) + 1)$ with $B( p_0, m, k) = \binom{m-1}{k-1} + |p_0| \binom{m-1}{k}$.
\end{proof}

\subsubsection{Existence of an expansive polynomial over a nonirreducible 
Salem polynomial or a cyclotomic polynomials}
\label{sec:SS4.3S2}

We now extend the existence Theorem \ref{main1} 
to monic expansive polynomials
lying over nonirreducible Salem polynomials or
cyclotomic polynomials only. The proofs, of similar nature as 
in \S \ref{sec:SS4.1S2} and in \S \ref{sec:SS4.2S2}, 
are left to the reader: they include a Criterion of expansivity,  
as a first step (Theorem \ref{critt2.4.3}), then imply the nonemptyness of a certain admissible 
open cone of solutions in a suitable Euclidean space. 
Let us remark that the construction method which is followed 
calls for expansive polynomials of even degrees, and for 
a LHS term in \eqref{eqSaProdCyclo} (as in \eqref{PtoT}) 
which has factors of multiplicity one
in its factorization (case of simple roots).    

\begin{theorem}
\label{critt2.4.3}
Let $T \in \zb[z]$, with simple roots 
$\not\in \{\pm 1\}$, of degree $m \geq 4$,
be either a cyclotomic polynomial or a Salem polynomial.
Let $P \in \zb[z]$ of degree $m$ such that 
\begin{equation}
\label{eqSaProdCyclo}
   (z-1)\, T(z) =  z P(z) - P^*(z). 
\end{equation}
\noindent
Then $P$ is an expansive polynomial if and only if
 \begin{itemize}
  \item[(i)]  $\displaystyle P(1) T(1) < 0$,
  \item[(ii)] for every zero $\alpha$ of $T$ of modulus $1$,
    $$(\alpha-1)\alpha^{1-m}P(\alpha)T'(\alpha)~~ \mbox{is real and negative}.$$
 \end{itemize}
\end{theorem}

Theorem \ref{main1} is a consequence of
Theorem \ref{critt2.4.3}.
 
\subsubsection{The second class of Salem numbers in $A_q$}
\label{sec:SS4.4S2}

Given an integer $q \geq 2$, the second class of 
Salem numbers $\beta$ in $A_q$ is provided by
expansive polynomials $P$ for which $P(0)=-$M$(P) = -q$ lying above the
minimal polynomial $T$ of $\beta$, referring to Theorem \ref{theoremA}. 
From \eqref{EqDeQ1} this case corresponds to $\epsilon = +1$, with 
$P_{1}(z) = (z-1)P(z)$, and then to
\begin{equation}
\label{niouEQ}
Q(z)=Q_{1}(z) = z P(z) + P^{*}(z).
\end{equation}
Since the RHS of \eqref{niouEQ} vanishes at $z=-1$ when $\deg(P)$ is even, 
a natural factorization
of $Q$ in \eqref{niouEQ} is $Q(z)=(z+1) Q_{2}(z)$, with $Q_2$ a Salem polynomial of even degree.
Then, instead of the equation \eqref{EqDeQ1}, this second class of Salem numbers
calls for association theorems between $Q_2$ and $P$ using
\begin{equation}
\label{niouEQfactor}
(z+1) Q_{2}(z) = z P(z) + P^{*}(z).
\end{equation} 
The association Theorem \ref{caseTWO} between a 
Salem polynomial and an expansive polynomial
of even degree
we obtain in this case can be deduced from the preceding ones, 
by suitable sign changes, as analogues of Theorem \ref{main1}.  

\begin{theorem}
\label{caseTWO}
Let $T \in \zb[z]$, with simple roots
$\not\in \{\pm 1\}$ of degree $m \geq 4$, be either 
a cyclotomic polynomial or a Salem polynomial.
Then there exists a monic expansive polynomial
$P \in \zb[z]$ of degree $m$ such that
\begin{equation}
\label{main1eqplus}
(z+1) T(z) = z P(z) + P^{*}(z).
\end{equation}
\end{theorem}

\noindent
{\bf Notation:} given a Salem polynomial $T(z) = t_m z^m + t_{m-1} z^{m-1}+\ldots+ t_{1} z + t_0$, of degree $m$, 
with $t_{m/2 +1}=t_{m/2 -1}, \ldots, t_{m-1} = t_1 , t_0 = t_m = 1$, we denote by
\begin{equation}
\label{omega1}
\omega_1 : \mbox{}^{t}(t_0 , t_1 , \ldots, t_{m-1}) ~\mapsto~ \mbox{}^{t}(p_0 , p_1 , \ldots, p_{m-1})
\end{equation} 
the (``choice") mapping sending the coefficient vector
of $T$ to that of the monic expansive polynomial 
$P(z) = p_m z^m + p_{m-1} z^{m-1}+\ldots+p_{1} z + p_0$ over it, $p_m = 1$, 
with $P(0)=p_0$ negative or positive, which is chosen such that the half coefficient vector lies 
on the lattice $\zb^{m/2}$ inside
the admissible cone in $\rb^{m/2}$. In this definition we exclude 
the leading coefficients, equal 
to 1. The converse mapping
\begin{equation}
\label{Omega1}
\Omega_1 :  
\mbox{}^{t}(p_0 , p_1 , \ldots, p_{m-1})  ~\mapsto~ \mbox{}^{t}(t_0 , t_1 , \ldots, t_{m-1}), 
\end{equation}
defined by 
$(z-1) T(z) = z P(z) - P^{*}(z)$
if $P(0)=p_0 > 0$, 
or 
$(z+1) T(z) = z P(z) + P^{*}(z)$
if $P(0)=p_0 < 0$, allows to sending 
a monic polynomial $P$ with real coefficients to
a monic polynomial $T$. For any Salem polynomial $T$ with coefficients $t_i$
as above, the identity
$\Omega_1 (\omega_1 (\mbox{}^{t}(t_i))) = \mbox{}^{t}(t_i)$ holds.

\subsubsection{Algorithmic determination of expansive polynomials}
\label{sec:SS4.5S2}

Let $T$ be an irreducible Salem polynomial of degree $m \geq 4$. 
Denote by $A$ the invertible matrix 
$\displaystyle A=(a_{i,j})_{0 \leq i, j \leq \frac{m}{2}-1}$, where the coefficients
$a_{i,j}$ are given by the equations of the face hyperplanes of the admissible cone
in $\rb^{m/2}$, namely \eqref{P1T1facette} and \eqref{ALPHAfacette}. 
Denote by $\Lambda$ be the column vector of 
$\rb ^{\frac{m}{2}}$ defined by $\Lambda=\mbox{}^t(\lambda_0 , \lambda_1 , \ldots , \lambda_{\frac{m}{2}-1 }) $,
where $\lambda_j$ is the constant term of the equation of
the $j$-th face hyperplane, given by \eqref{P1T1facette} and \eqref{ALPHAfacette}. 
The polyhedral admissible cone is defined by the pair
$(A, \Lambda)$; the pair $(A, \Lambda)$ is said to be {\it associated with} $T$. 

For any vector $V$ in $\rb^{m/2}$ let us denote by $\mbox{round}(V)$ the vector of 
$\zb^{m/2}$ the closest to $V$, componentwise 
(i.e. for each component $x$ of $V$ in the canonical basis of $\rb^{m/2}$, 
the nearest integer to $x$ is selected).   

 \begin{proposition}
   \label{Pexistsprop}
 Let $m\geq 4$ be an even integer. 
 Let $T(z) = \sum_{i=0}^{m} t_{m-i} z^i$ be an irreducible Salem polynomial of degree $m$, 
with associated pair $(A = (a_{i,j}), \Lambda = (\lambda_i))$. 
 Let 
\begin{equation}
\label{bconstant}
b := \frac{\sqrt{m}}{2 \sqrt{2}} \max_{0 \leq i \leq \frac{m}{2}-1} 
\Bigl(\sum_{k=0}^{\frac{m}{2}-1} a_{i,k}^2 \Bigr)^{1/2} \,.
\end{equation} 
For $t \geq b$ let $\mathcal{M}_t= \mbox{{\rm round}}( A^{-1} (^t(-t, -t, ..., -t) - \Lambda))
:=\mbox{}^{t}(M_0, \ldots, M_{\frac{m}{2}-1})$.
Then 

(i) the polynomial $P(z)=z^m + p_{m-1} z^{m-1} +...+ p_1 z + p_0$ defined by
\begin{equation}
\label{casesequations}
  \begin{cases}
     p_i=M_i , & \mbox{} \qquad 0 \leq i \leq \frac{m}{2}-1, \\
   p_{m-i}=M_{i-1}+t_i-t_{i-1}, & \mbox{} \qquad 1 \leq i \leq \frac{m}{2}-1,
  \end{cases}
\end{equation}
is an expansive polynomial which satisfies $(z-1) T(z) = z P(z) - P^{*}(z)$;

(ii) the Salem number $\beta$ of minimal polynomial $T$ belongs to
the class $A_{q_b}$ with
\begin{equation}
\label{q_calculexplicite}
q_b = M_0 = \mbox{pr}_0 \mathcal{M}_{b}  
\end{equation} 
possibly to other classes $A_q$ with $q < q_b$.
\end{proposition}
  
\begin{proof}
(i) For $t \in \rb^{+}$, let $\mathcal{M}_t = \mbox{}^{t}({M}_0, {M}_1, ..., {M}_{\frac{m}{2}-1})$,
with real coordinates in the canonical basis of $\rb^{m/2}$,
be the unique vector solution of the equation
\begin{equation}
\label{coordonneesMt}
A \mathcal{M}_t+\Lambda= \mbox{}^{t}(-t, -t, \ldots, -t).
\end{equation}
If $t \neq 0$, 
$\mathcal{M}_t$ lies in the open admissible cone  
and, if $t=0$, $\mathcal{M}_0$ is the summit vertex of the cone.
Let $d_2$ be the usual Euclidean distance on $\rb^{m/2}$
and $d_{\infty}$ the distance defined by 
$d_{\infty}(V,W) = \max_{i=0,\ldots, m/2 -1} |V_i - W_i|$. Denote by
$\overline{B_{\infty}}$ the closed balls (cuboids) for this distance. 

For $0 \leq i \leq \frac{m}{2}-1$ denote by $\emph{H}_i$ the $i$-th face hyperplane
of the admissible cone, by   
$\mathcal{N}_i=\mbox{}^{t}(a_{i,0},a_{i,1} , ... , a_{i,\frac{m}{2}-1})$ the vector orthogonal
to the hyperplane $\emph{H}_i$ and by pr$_{H_i}(\mathcal{M}_{t})$ the orthogonal projection
of $\mathcal{M}_{t}$ on $H_i$. 
For all $t \geq b$, the scalar product
$\|\mathcal{N}_i\|^{-1} (\mathcal{N}_{i} \cdot \mathcal{M}_t)$ is equal to
 $d_{2}(\mathcal{M}_t,\mbox{pr}_{\emph{H}_i}({\mathcal{M}}_{t}))
 =$
$$\frac{|a_{i,0} {M}_0 + a_{i,1}{M}_1 +  ...  a_{i,\frac{m}{2}-1}{M}_{\frac{m}{2}-1} + \lambda_i|} {\sqrt{a_{i,0}^2 + a_{i,1}^2 + ... + a_{i,\frac{m}{2}-1}^2}} 
  > \frac{|-t|}{b  \frac{2 \sqrt{2}}{\sqrt{m}}} \geq \frac{\sqrt{m}}{2 \sqrt{2}}.$$
Therefore the cuboid $\overline{B}_{\infty}(\mathcal{M}_{t},\frac{1}{2})$, which
 is a fundamental domain of the lattice  $\displaystyle \zb^\frac{m}{2}$, does not contain the point
pr$_{\emph{H}_i}(\mathcal{M}_t)$. This cuboid is included in the (open) admissible cone $\emph{C}^{+}$. 
Consequently the vector $\displaystyle (p_i)_ {0 \leq i \leq  {\frac{m}{2}-1}} = \mbox{round}(\mathcal{M}_{t})$ lies 
in $ \zb^ {\frac{m}{2}} \cap \emph{C}$; it is the half-coefficient vector of an expansive polynomial.

(ii) The above method allows us to explicitely compute
the expansive polynomial $P$ the closest to the summit vertex 
$\mathcal{M}_{0}$ of the cone:
by solving \eqref{coordonneesMt} with $t=b$ and using
\eqref{casesequations}. Since the integer component $M_0$ 
of $\mathcal{M}_{b}$ is equal to the Mahler measure $q$ of $P$
and that $P$ produces $\beta$, we deduce the claim.
\end{proof}

Similar explicit bounds $q_b$ for $q$ are easy to establish if $T$ is
a nonirreducible Salem polynomial (as in \S \ref{sec:SS4.3S2})
or in the second case of Salem numbers in $A_q$ (as in \S \ref{sec:SS4.4S2})
using the equations of the face hyperplanes
deduced from the Criterions of expansivity. 
We do not report them below.

We have seen that, using  Theorem \ref{theoremA} and Theorem \ref{mcsm51}, 
we can construct sequences of Salem polynomials and Salem numbers, 
depending upon the choice of the expansive polynomial over the Salem polynomial. 
With the method of Theorem \ref{Pexistsprop}, we compute iteratively for instance
the following sequence of Salem numbers from the Lehmer number:
$\theta \approx 1.1762808$, 
$\theta_2\approx 8.864936672$ , $\theta_3\approx 21.56809548$, $\theta_4\approx 45.44171097 $, 
$\theta_5\approx 87.36603624$, $\theta_6\approx 155.3214209$, $\theta_7\approx 261.2942782$, 
$\theta_8\approx 423.2784753$, $\theta_9\approx 668.2676464$, $\theta_{10}\approx 1037.261121$.
  
The Schur-Cohn-Lehmer Theorem may be used 
to compute an expansive polynomial $P$ 
over a Salem polynomial $T$ with dominant root in $A_q$
for a given $q \leq q_b$. It may happen that
this component-by-component search for the coefficient vector of $P$ fails.

\begin{theorem}[Schur-Cohn-Lehmer] 
Let $P(z) = p_m z^m + p_{m-1} z^{m-1} + .. + p_1 z + p_0 , 
~p_m p_0 \neq 0,$ 
be a polynomial of degre m with real coefficients.
Let $\mathcal{T}$ be the transformation 
$\mathcal{T} : P(z) \mapsto \mathcal{T}(P)(z) = p_0 P(z) - p_m P^*(z)$. 
Let $\mathcal{T}^{1} := \mathcal{T}$ and 
$\mathcal{T}^{k} := \mathcal{T} (\mathcal{T}^{k-1})$ for
$k \geq 2$. Then
$P(z)$ is an expansive polynomial if and only if 
$$\mathcal{T}^{k}(P)(0) > 0 \qquad \quad \mbox{for all}~ 1 \leq k \leq m .$$
\end{theorem}
\begin{proof}
Lehmer \cite{lehmer}, Marden \cite{marden}, pp 148-151.
\end{proof}

Note that the first step of the algorithm gives to  
$\mathcal{T}(P)(0)=p_0 ^2 -p_{m} ^2 = p_0 ^2 - 1 > 0$, 
for a monic expansive polynomial $P$,
so that the inequality $q = |p_0| =
$ M$(P) \geq 2$
holds (cf Remark \ref{qdessus2}).  
  
The coefficient vector $(p_i)$ of $P$ can be 
constructed recursively: for $k > 0$,
 knowing $(p_i)_{i \leq k-1}$ and $(p_i)_{ i \geq 2n-k+1}$, 
 we choose $p_k$ in the interval given by the quadratic equation in $p_k$ : $\mathcal{T}^{k+1}(P)(0)>0$ and
 we compute $p_{m-k} = -t_{k-1}+t_k + p_{k-1}$. 
 For some initial values $p_0=q$ not large enough, we obtain no solution. 

The 20 smallest Salem numbers, with their minimal polynomials and associated 
expansive polynomials are presented in Figure \ref{tableau_1}.

\begin{figure}
\begin{center}
\mbox{} \hspace{1.5cm}
\includegraphics[width=11.75cm]{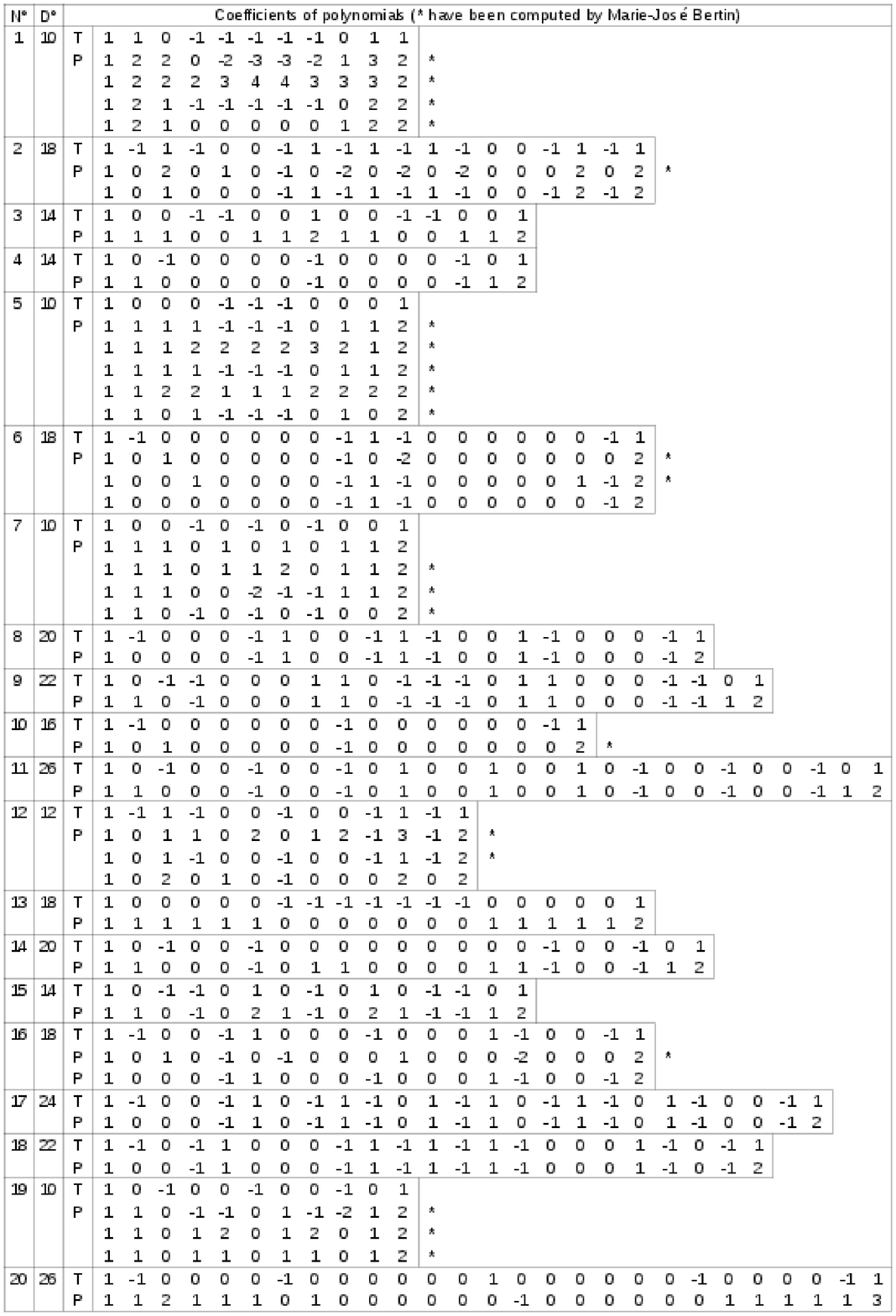}
\end{center}
\caption{Minimal polynomials $T$ of the 20 smallest Salem numbers $\beta$ 
and some of the associated monic expansive polynomials $P$ over $T$.
The degree $D^{°}$ is $\deg (T) = \deg (P)$. 
The coefficients of $P$ are such that the constant term, on the 
right, is the integer $q \geq 2$ for which $\beta \in A_q$.}  
\label{tableau_1}
\end{figure}

\begin{remark}
\label{omega1_CetT}
Let $T(z)$ be a Salem polynomial, of degree $m$. 
Denote by $k$ the maximal real
subfield of the splitting field of $T$. The
admissible cone $\emph{C}^{+}$ is defined by the associated pair
$(A = (a_{i,j}), \Lambda = (\lambda_i)) 
\in GL(\frac{m}{2}, k) \times k^{m/2}$, i.e.  
the affine equations of its facets, in
$\mathbb{R}^{m/2}$, which are 
functions of the roots of $T$.
Every point, say $\uP$, of
$\emph{C}^{+} \cap \zb^{m/2}$ has coordinates
$\mbox{}^{t}(p_0, p_1, \ldots, p_{m/2 - 1})$ which 
constitute the half coefficient vector of 
a unique monic expansive polynomial, say $P$, 
by \eqref{lienPiTi}:
there is a bijection between
$\uP$ and $P$. 
Denote by $\widehat{\uP}$ the second half 
of the coefficient vector of $P$.
The set of monic
expansive polynomials $P$, defined by their coefficient vectors,
is an infinite  subset of $\zb^m$,
in bijection with $\emph{C}^{+} \cap \zb^{m/2}$.
The map
$\omega_1$ decomposes into two parts:
$$\omega_{1,C^+}: \zb^m \to \emph{C}^{+} \cap \zb^{m/2}, 
\mbox{}^{t}(t_i)_{i=0,\ldots,m-1} \to \mbox{}^{t}(p_i)_{i=0,\ldots,\frac{m}{2}-1}$$
and, once the image of $\omega_{1,C^+}$ is fixed, using \eqref{lienPiTi},
$$\widehat{\omega_{1,C^+}}: \zb^{m/2} \to \zb^{m/2}, 
\mbox{}^{t}(p_i)_{i=0,\ldots,\frac{m}{2}-1} \to (p_{\frac{m}{2}+i})_{i=0,\ldots,\frac{m}{2}-1}.$$
By abuse of notation, we denote: 
\begin{center}
$\displaystyle \omega_{1,C^+}(T) = \uP ,
~\widehat{\omega_{1,C^+}}(\uP) = \widehat{\uP}$,
~and~ $\omega_{1}(T) = (\uP, \widehat{\uP}) = P.$  
\end{center}
\end{remark}

\subsubsection{Condition on an expansive polynomial P to produce a Salem number or a negative Salem number}
\label{sec:SS4.6S2}

Given a monic expansive polynomial $P(X) \in \zb[X]$ 
we separate two cases for the geometry of the roots of $Q$
by the $Q$-construction (cf \S \ref{sec:SS1S2}):
(i) either all the zeroes of $Q$ are on $|z| = 1$ : $Q$ is then
a product of cyclotomic polynomials, or
(ii)
$Q$ has all but two zeroes on $|z|=1$, and we obtain a Salem number $\beta$.

To ensure that $Q$ has a root outside the closed unit disk 
Boyd (\cite{boyd0} p.318, Corollary 3.2) proposes a criterion 
counting the number of entrances of $Q$. Below we give a criterion on $P$
to discriminate between the two cases, in a more general setting 
including polynomials vanishing at negative 
Salem numbers \cite{haremossinghoff}.
\begin{proposition}
Let P be a monic expansive polynomial, with integer coefficients,
of even degree $m \geq 4$. Let $\eta =+1$ or $\eta=-1$.
Let $T$ be a polynomial having simple roots which satisfies 
$$(z + \eta) T(z) = z P(z) + \eta P^*(z).$$
If the following inequality 
\begin{equation}
\label{Pcrit1}
P(\eta)((1-m)P(- \eta) - 2 \eta P'(- \eta)) < 0
\end{equation} 
holds then 
$T(- \eta \, {\rm sgn}(P(0)) \, z) $ 
is a Salem polynomial which does not vanish at $\pm 1$. 
\end{proposition}
\begin{proof}
Let us give a proof with $\eta = -1$ and $P(0)> 0$, the other cases
being treated using similar arguments.
 
First, $T$ is necessarily monic, reciprocal, with integer coefficients.
Now, deriving $(z-1)T(z)=zP(z)-P^*(z)$, we obtain: 
$(z-1)T'(z)+T(z)=P(z) + zP'(z) - m z^{m-1} P(\frac{1}{z}) + z^{m-2} P'(\frac{1}{z})$.
Hence $T(1)=(1-m)P(1) + 2 P'(1)$. 
On one hand, $T(-1) = P(-1) \neq 0$ since $P$ is expansive.
Let us show that $T(1) \neq 0$. Let us assume 
the contrary. Since $T$ has even degree and that any cyclotomic polynomial
$\Phi_n$ has even degree as soon as $n > 2$, 
an even power of $(z-1)$ should divide $T(z)$. 
Since $T$ has simple roots we obtain a contradiction. 
By the theorem of intermediate values, 
$T$ has a unique real root in $(-1, 1)$ if and only 
if the condition $T(1)T(-1) < 0$, equivalently
\eqref{Pcrit1}, holds. 
\end{proof}

\subsubsection{Semigroup structure}
\label{sec:SS4.7S2}

Let $q \geq 2$ be an integer. Let $m \geq 4$ be an even integer.  
Let $T$ be a Salem polynomial of degree $m$, vanishing at 
a Salem number $\beta$, having simple roots, such that $T(\pm 1) \neq 0$.
Denote by $E_{T,q}^{+} \subset E_{q}$ the set of monic expansive polynomials 
over $T$ of degree $m$, defined by their coefficient vectors,
as a point subset of the lattice $\zb^m$, satisfying \eqref{PtoT}, with 
$P(0)=+q$, resp. $E_{T,q}^{-} \subset E_{q}$ the set of monic expansive polynomials
over $T$ of degree $m$ satisfying \eqref{niouEQfactor}, with
$P(0)=-q$. 
Denote by $\mathcal{P}_{T}^{+} :=
\bigcup_{q \geq 2} E_{T,q}^{+}$, resp.  
$\mathcal{P}_{T}^{-} :=
\bigcup_{q \geq 2} E_{T,q}^{-}$
and
$$\mathcal{P}_{T} := \mathcal{P}_{T}^{+} \cup \mathcal{P}_{T}^{-} =
\bigcup_{q \geq 2} \Big( E_{T,q}^{+} \cup E_{T,q}^{-} \Bigr).$$

\begin{theorem}
 Let $P \in E_{T,q}^{+}$  and $P^{\dag} \in E_{T,q^{\dag}}^{+}$ be two monic expansive polynomials 
over $T$ of degree $m$.
 Then the sum $P+P^{\dag}-T$ is a monic expansive polynomial of degree $m$ over $T$,
in $E_{T,q+q^{\dag}-1}^{+}$. 
The internal law $(P, P^{\dag}) \mapsto P \oplus P^{\dag} := P+P^{\dag}-T$ defines 
a commutative semigroup structure on $\mathcal{P}_{T}^{+} \cup \{T\}$, 
where $T$ is the neutral element. 
\end{theorem}

\begin{proof} 
Let us give two proofs. 
(1) Using McKee-Smyth interlacing quotients:
by Theorem \ref{theoremA}, 
the polynomials $L:=P-T$ and $L^{\dag}:=P^{\dag}-T$ are monic, reciprocal and satisfy the conditions: 
\begin{itemize}
 \item $L(0)=q-1$ and $L^{\dag}(0)=q^{\dag}-1$,
 \item $\deg(L) = \deg(L^{\dag}) = \deg(T)-1$,
 \item $L(1) \leq -T(1)$ and $L^{\dag}(1) \leq -T(1)$,
 \item the zeroes of $L$ interlace the zeroes of $T$ on the half unit circle and similarly
for the zeroes of $L^{\dag}$.
\end{itemize}
After Definition \ref{CSentrecroisement} 
in \S \ref{sec:SS3S2}, with $K=K^{\dag}=1$, 
$\displaystyle \frac{(z-1)L}{T}$ and $\displaystyle \frac{(z-1)L^{\dag}}{T}$ are CS-interlacing
quotients.
As in the proof of McKee and Smyth's Proposition 6 
in \cite{mckeesmyth0},
we can transform these quotients into real quotients by the Tchebyshev 
transformation $x=z+\frac{1}{z}$ and 
we obtain in the real world the rational function 
$\displaystyle f(x)=\gamma \, \frac{\prod(x-\delta_i)}{\prod(x-\alpha_i)}$ 
and respectively 
$\displaystyle f^{\dag}(x)=\gamma^{\dag} \, \frac{\prod(x-\delta^{\dag}_i)}{\prod(x-\alpha_i)}$. 
These rational functions can be written 
$\displaystyle f(x)=\sum \frac{\lambda_i}{(x-\alpha_i)}$ 
and 
$\displaystyle f^{\dag}(x)=\sum \frac{\lambda^{\dag}_i}{(x-\alpha_i)}$ 
with $\lambda_i > 0$ and $\lambda^{\dag}_i>0$. 
Then $\displaystyle (f+f^{\dag})(x) =\sum \frac{\lambda_i + \lambda^{\dag}_i}{(x-\alpha_i)}$ 
has strictly positive coefficients. Its derivative is everywhere negative. 
Then there is exactly one zero between two successive poles. 
Then, back to the complex world, 
$\displaystyle \frac{(z-1)(L+L^{\dag})}{T}$ is a CS-interlacing quotient.
Recall that $T(1) < 0$ since $T(1/\beta) = T(\beta) = 0$. 
Finally, we have \begin{itemize}
                  \item $(L+L^{\dag})(0) = (q+q^{\dag}-1)-1$,
                  \item $\deg(L+L^{\dag}) = \deg(T)-1$ because $L$ and $L^{\dag}$ are both monic,
                  \item $(L+L^{\dag})(1) \leq -2T(1) \leq -T(1)$ because $T(1) <0$,
                  \item the zeroes of $L+L^{\dag}$ and the zeroes of $T$ interlace on the half unit circle.
                 \end{itemize}
Thus $L+L^{\dag}+T  \in E_{T,q+q^{\dag}-1}^{+}$.

(2) Using the equations of the face hyperplanes of the admissible cone 
$\emph{C}^+$ associated with $T$:
using the notation of \S \ref{sec:SS4.5S2}
and Remark \ref{omega1_CetT}, all the components of the vectors
$A \uP + \Lambda$ and 
$A \uP^{\dag} + \Lambda$ are strictly negative and
$$A \uT + \Lambda = \mbox{}^{t}((T(1))^{2}, 0, \ldots, 0).$$
Then all the components of $ A \bigl(\uP + \uP^{\dag} - \uT\bigr) +\Lambda$
are strictly negative. 
Hence $\uP + \uP^{\dag} - \uT$ belongs to
$\emph{C}^{+}$. Thus $L+L^{\dag}+T  \in E_{T,q+q^{\dag}-1}^{+}$.
\end{proof}

The semigroup law $\oplus$ satisfies: 
  \begin{itemize}
   \item $T \oplus T = T$,
   \item $P \oplus P^{\dag} = P+L^{\dag}= P^{\dag}+L = L+L^{\dag}+T$,
   \item $(P \oplus P^{\dag}) -T = (P-T) + (P^{\dag}-T) $. 
  \end{itemize}
Note that the opposite (additive inverse) of $P \in E_{T,q}^{+}$ is
$2 T - P$, which is outside the set $\mathcal{P}_{T}^{+}$ 
since $2 T(0) - P(0) = 2-q \leq 0$. 
Thus $(\mathcal{P}_{T}^{+} \cup \{T\}, \oplus)$ 
is a semigroup and not a group.

The structure of the set of positive real generalized Garsia numbers, in particular
for those of Mahler measure equal to $2$ (i.e. Garsia numbers in the usual sense), 
is of interest for many purposes
\cite{harepanju}. In the present context of association 
between Salem numbers and generalized Garsia numbers,
the above semigroup structure can be transported to sets
of generalized Garsia numbers and generalized Garsia polynomials as follows. 

\begin{corollary}
If $P$ and $P^{\dag}$ are two generalized 
Garsia polynomials of the same degree $m \geq 4$ over a given
Salem polynomial $T$, such that $P(0)=+ {\rm M}(P), P^{\dag}(0) = + {\rm M}(P^{\dag})$, satisfying the assumption
that $P(0) + P^{\dag}(0) -1$ is a prime number, then 
$P + P^{\dag} - T$ is a generalized Garsia polynomial of degree $m$ 
and Mahler measure equal to
${\rm M}(P) + {\rm M}(P^{\dag}) -1$.
\end{corollary}

\section{From expansive polynomials to Hurwitz polynomials, Hurwitz alternants, and continued fractions}
\label{S3}

\subsection{Hurwitz polynomials }
\label{sec:SS1S3}

\begin{definition}
A {\it real Hurwitz polynomial} $H(X)$ is a polynomial in $\rb[X]$ of degree $\geq 1$ 
all of whose roots have negative real part. 
\end{definition}

The constructions involved in Theorem \ref{theoremA} 
require monic expansive polynomials with coefficients in
$\zb$.  
Since the conformal map $\displaystyle z \mapsto Z=\frac{z+1}{z-1}$
transforms the open unit disk $|z|<1$ to the half plane $Re(Z)<0$, it transforms the
set of zeros of the reciprocal polynomial $P^*$ of 
an expansive polynomial $P$ into the set of zeros of
a Hurwitz polynomial.
In other terms the polynomial $P(z)=p_0 + p_1 z + \ldots +p_m z^m$ is expansive if and only 
if the polynomial $\displaystyle H(z)=(z-1)^m P^*(\frac{z+1}{z-1})$ is a real Hurwitz polynomial. 
In this equivalence, note that $z=0$ is never a root of $P^*$, 
implying that $H$ never vanishes at $z=-1$.
 The analytic transformation $z \to Z$, resp. $Z \to z = (1+Z)/(Z-1)$, 
corresponds to a 
linear transformation of the coefficient vector 
of $P$, resp. of $H$, of $\rb^{m+1}$, as follows, 
the proof being easy and left to the reader.
 
\begin{proposition}
\label{PPtoHHtoPP}
\begin{itemize}
\item[(i)] If $P(z)=p_0+p_1z+...+p_m z^m$ is an expansive polynomial 
then $\displaystyle H(z)=(z+1)^m P(\frac{z-1}{z+1})$ is a real Hurwitz polynomial 
and, denoting $H(z)=h_0+h_1 z+...+h_m z^m$, 
then $h_0 \neq 0$ and
\begin{equation}
\label{PPtoHH}
h_i = \sum_{j=0}^{m} \sum _{k=0}^i  (-1)^{j-k} \binom{j}{k} \binom{m-j}{i-k}  \,p_j , 
\qquad 0\leq i \leq m.
\end{equation}
\item[(ii)] If $H(z)=h_0+h_1z+...+h_m z^m$ is a real Hurwitz polynomial 
such that $H(-1) \neq 0$, 
then $\displaystyle P(z)=(1-z)^m H(\frac{z+1}{1-z})$ is an expansive polynomial 
and, if we put $P(z)= p_0+p_1 z+...+p_m z^m$, then 
\begin{equation}
\label{HHtoPP}
p_i = \sum_{j=0}^{m} \sum _{k=0}^i  (-1)^{i-k} \binom{j}{k} \binom{m-j}{i-k}  \,h_j , 
\qquad 0\leq i \leq m.
\end{equation}
\end{itemize}
\end{proposition}

\noindent
{\bf Notation:} 
using the formulae \eqref{PPtoHH}, since
$h_0 = \sum_{i=0}^{m} (-1)^i p_i = P(-1) \neq 0$, we
denote by
\begin{equation}
\label{omega2}
\omega_2: \mbox{}^{t}(p_0 , p_1 , \ldots , p_{m-1}) \mapsto 
\mbox{}^{t}\Bigl(\frac{h_1}{P(-1)}, \frac{h_2}{P(-1)}, \ldots , \frac{h_m}{P(-1)}\Bigr)
\end{equation}
the mapping which sends the coefficient vector of
the monic polynomial 
$P(z) = p_0+p_1z+...+ z^m$
to that of $H(z)=(P(-1))^{-1} 
\bigl(h_0 +h_1 z+...+h_m z^m\bigr)$ where $P(-1) = \sum_{i=0}^{m-1} (-1)^i p_i + 1$
here; let us also remark that the constant term 
$h_0 / P(-1)$ of this polynomial is equal to $1$.

Conversely, using \eqref{HHtoPP}, since
$p_m = \sum_{i=0}^{m} (-1)^i h_i = H(-1) \neq 0$, we denote by
\begin{equation}
\label{Omega2}
\Omega_2: \mbox{}^{t}(h_1, h_2, \ldots , h_m) \mapsto 
\mbox{}^{t}\Bigl(\frac{p_0}{H(-1)}, \frac{p_1}{H(-1)}, \ldots , \frac{p_{m-1}}{H(-1)}\Bigr)
\end{equation}
the mapping which sends the coefficient vector of
the Hurwitz polynomial $H = 1 + h_1 z+...+h_m z^m$,
having constant coefficient equal to 1, to that of the
polynomial $P(z) = (H(-1))^{-1} \bigl(p_0+p_1z+ \ldots + p_m z^m\bigr)$
where $H(-1) = 1 + \sum_{i=1}^{m} (-1)^i h_i$ here; 
let us remark that the leading coefficient $p_m / H(-1)$ is equal to $1$.
Reversing the order of the terms, we denote by
\begin{equation}
\label{Omega2star}
\Omega^{*}_2: \mbox{}^{t}(h_1, h_2, \ldots , h_m) \mapsto 
\mbox{}^{t}\Bigl(\frac{p_{m-1}}{H(-1)}, \frac{p_{m-2}}{H(-1)}, \ldots , \frac{p_{0}}{H(-1)}\Bigr)
\end{equation}
the mapping which sends the coefficient vector of
the Hurwitz polynomial $H = 1 + h_1 z+...+h_m z^m$ 
to that of the polynomial 
$P^{*}(z) = 1+(H(-1))^{-1} \bigl(p_{m-1} z + \ldots + p_0 z^m\bigr)$
with $H(-1) = 1 + \sum_{i=1}^{m} (-1)^i h_i$.

Hence, the transformation $\Omega_2 \,o\, \omega_2$ is 
the identity transformation on the set of monic expansive polynomials
$P$.

\begin{remark}
\label{remm01}
The transformations \eqref{PPtoHH} and \eqref{HHtoPP} 
are $\zb$-linear.
Hence, for any real field $\mathbb{K}$
and any expansive polynomial $P(z) \in \mathbb{K}[z]$,
the projective classes 
$\mathbb{K} \bigl(\mbox{}^{t}(p_i)\bigr)$ and $\mathbb{K} \bigl(\mbox{}^{t}(h_i)\bigr)$ are
in bijection. This observation 
justifies the definitions
of $\omega_2$ and $\Omega_2$. Indeed, these mappings suitably send a
representative of $\mathbb{K} \bigl(\mbox{}^{t}(p_i)\bigr)$ to a representative
of $\mathbb{K} \bigl(\mbox{}^{t}(h_i)\bigr)$, which allows us to fix one coefficient
in the coefficient vector: 
in the first case the leading coefficient
$p_m$ equal to 1, in the second case the constant term $h_0$ equal to 1.
\end{remark}

\begin{theorem}  [Stodola, 1893]
 If $H(z)=h_0+h_1 z + h_2 z^2 + .. +h_m z^m \in \mathbb{R}[z]$ is a Hurwitz polynomial 
 then $(h_i)_{0 \leq i \leq m}$ are either all positive or all negative.
\end{theorem}

\begin{proof} 
 Let $r_1, r_2, .., r_p$ be the real zeros of a real polynomial $H(z)$ and 
 let $\alpha_1, \alpha_2,\ldots, \alpha_n, \overline{\alpha_1}, \overline{\alpha_2},\ldots, \overline{\alpha_n}$ 
be its nonreal zeros.
 Then we have the factorization : 
 $$H(z)=h_m \prod_{1 \leq k \leq p} (z-r_k) \prod_{1 \leq j \leq n} (z-\alpha_j)(z-\overline{\alpha_j})$$
 $$H(z)=h_m \prod_{1 \leq k \leq p} (z-r_k) \prod_{1 \leq j \leq n} (z^2-2 \operatorname{Re}(\alpha_j)z + |\alpha_j|^2)$$
 Now, $H(z)$ is a Hurwitz polynomial if $r_k <0$, for all $1 \leq k \leq r$, and 
$\operatorname{Re}(\alpha_j)<0$, for all $1 \leq j\leq n$. 
Then all the coefficients of all the factors are positive. 
Thus, if $h_m >0$, all the coefficients of $H(z)$ are positive and
 if  $h_m<0$, then they are all negative.
\end{proof} 

\begin{remark}
 From now on, we consider that the coefficients $h_i$ of the real Hurwitz polynomials are all 
(strictly) positive 
(otherwise, we will consider $-H(z)$).
\end{remark}

\subsection{Hurwitz alternants and Stieltjes continued fractions }
\label{sec:SS2S3}

\begin{definition}
\label{hurwitzquotientDEF}
The {\it quotient (alternant) of Hurwitz} $h(z)$ associated with a polynomial 
$H(z)=h_m x^m+h_{m-1} x^{m-1} + \ldots + h_1 x +h_0 \in \mathbb{R}_{> 0}[z]$ is the rational
function 
\begin{equation}
\label{hurquo}
h(z):=\frac{h_1+h_3z+h_5z^2+..}{h_0+h_2z+h_4z^2+..}= 
\frac{\sum_{i=0}^{\lfloor{(m-1)/2}\rfloor} h_{2i+1} z^i }{\sum_{i=0}^{\lfloor{m/2}\rfloor} h_{2i} z^i}.
\end{equation}
\end{definition}

The continued fraction of the quotient of Hurwitz $h(z)$ 
of \eqref{hurquo} is
$$h(z)=\cfrac{h_1+h_3z+h_5z^2+..}{h_0+h_2z+h_4z^2+..}
      =\cfrac{ \cfrac{h_1}{h_0}  (1+\cfrac{h_3}{h_1} z + \cfrac{h_5}{h_1}z^2 + ... )} 
           {1+\cfrac{h_2}{h_0} z + \cfrac{h_4}{h_0}z^2 + ... }$$
$$= \cfrac{h_1 / h_0}{1 + z \left[ \cfrac{ (h_1 h_2-h_0 h_3) + (h_1 h_4-h_0 h_5) z + (h_1 h_6-h_0 h_7) z^2...}{h_0 h_1 + h_0 h_3 z+ h_0 h_5 z^2+ ...} \right] }$$
$$= \cfrac{\frac{h_1}{h_0}}{1 +  \cfrac{\frac{h_1 h_2-h_0 h_3}{h_0 h_1} \, z}{1 + z \left[...\right] }} 
~=~ \cfrac{f_1}{ 1+ \cfrac {f_2 z}{ 1+\cfrac{f_3 z}{ 1+ \cfrac{...}{1+\ldots}}}} \quad {\rm written} 
=: \big[f_1 / f_2 / f_3 / \ldots\bigr](z)$$
for short, with $f_1 = h_1 / h_0$, $f_2 = (h_1 h_2-h_0 h_3)/(h_0 h_1)$, $\ldots$ 
Stieltjes \cite{stieltjes} (1894) has extensively studied these continued fractions, 
their Pad\'e approximants and the expressions of the numerators $f_j$ as
Hankel determinants (Henrici \cite{henrici}, Chap. 12, p 549 and 
pp 555--559, where 
they are called SITZ continued fractions).
For $H$ as in Definition \ref{hurwitzquotientDEF}, 
denote $D_0 := 1, D_{-1} := h_{0}^{-1}, D_{-2} := h_{0}^{-2}$, 
$$D_1 = h_1, 
~D_2 = \left|\begin{array}{cc}
h_1 & h_3\\
h_0 & h_2
\end{array}
\right|, 
~D_3 = \left|
\begin{array}{ccc}
h_1 & h_3 & h_5\\
h_0 & h_2 & h_4\\
0 & h_1 & h_3
\end{array}
\right|$$
and generally $$D_j := \det\bigl( h_{j}^{(1)}, h_{j}^{(3)}, \ldots, h_{j}^{(2 j - 1)} \bigr),$$ 
where $h_{j}^{(r)}$ denotes the 
$j$-dimensional column vector whose components are the 
first $j$ elements of the sequence
$h_r, h_{r-1}, h_{r-2}, \ldots$ ($h_r :=0$ for $r > m$ and for
$r<0$). Then
\begin{equation}
\label{fjHANKEL}
f_j = \frac{D_j}{D_{j-1}} \frac{D_{j-3}}{D_{j-2}}\,, \qquad \qquad j = 1, 2, \ldots, m .
\end{equation}
\begin{definition}
A rational function $F(z) \in \mathbb{R}(z)$ is a 
$m$-terminating continued fraction if it exist $f_1 , f_2 , \ldots, f_m \in \rb, f_m \neq 0$, such that
$F(z) = \big[f_1 / f_2 / \ldots / f_m\bigr](z)$.

\end{definition}
\begin{theorem} [Hurwitz, 1895]
\label{fipos}
The polynomial $H(z) \in \mathbb{R}[z]$, of degree $m \geq 1$, 
has all his roots on the left half plane $Re(z)<0$ if and only if 
the Hurwitz alternant $h(z)$ of $H(z)$ can be represented by an 
$m$-terminating continued fraction $\big[f_1 / f_2 / \ldots / f_m\bigr](z)$ in which every number $f_k$ is 
(strictly) positive.
\end{theorem}

\begin{proof} Henrici \cite{henrici}, Theorem 12.7c.
\end{proof}
If $h(z)$ is a $m$-terminating rational fraction, given by \eqref{hurquo}, 
then it is in its lowest terms, i.e. 
the numerator $\sum_{i=0}^{\lfloor{(m-1)/2}\rfloor} h_{2i+1} z^i$ 
and the denominator $\sum_{i=0}^{\lfloor{m/2}\rfloor} h_{2i} z^i$
of $h(z)$ have no common zeros (\cite{henrici}, Theorem 12.4a). 
Two real Hurwitz polynomials
$H_1$ and $H_2$ such that $H_1 / H_2 \in \rb^*$ have the same Hurwitz quotient by
\eqref{hurquo}.
The real Hurwitz polynomials $H$ obtained by 
$\omega_2$ (Proposition \ref{PPtoHHtoPP}) 
from expansive polynomial $P$ in $\zb\left[z\right]$  
have coefficients in $\zb$ and even degrees. Therefore 
their Hurwitz alternants are represented by $m$-terminating
continued fractions $\big[f_1 / f_2 / \ldots / f_m\bigr](z)$ with 
$f_j \in \qb_{>0}$ for $1 \leq j \leq m$, such that 
$h(1) \neq 1$ since $H(-1) \neq 0$.
Recall that $h(x) > 0$ for $x \geq 0$ in general, with
$h(\{{\rm Im}(z) > 0\}) \subset \{{\rm Im}(z)<0\}, 
h(\{{\rm Im}(z) < 0\}) \subset \{{\rm Im}(z)>0\}$
(\cite{henrici}, p 553).

In the case where all the $f_j$s are positive rational numbers,
following Stieltjes \cite{stieltjes}, the $m$-terminating continued fraction
$\big[f_1 / f_2 / \ldots / f_m\bigr](z)$ 
can also be written with integers, say $e_0 , e_1 , e_2 , \ldots, e_m$, as
$$\cfrac{e_1}{ e_0+ \cfrac {e_2 z}{ e_1+\cfrac{e_3 z}{ e_{m-2}+ \cfrac{...}{e_{m-2}+e_m z}}}}
=
\cfrac{\cfrac{e_1}{e_0}}{1+\cfrac {\cfrac{e_2 }{e_0 e_1}z}{ 1+  
\cfrac{...}{1+\cfrac{e_m}{e_{m-2}} z}}}$$
with $\displaystyle f_1=\frac{e_1}{e_0}$, $\displaystyle f_m=\frac{e_m}{ e_{m-2}}$ and 
$\displaystyle f_i=\frac{e_i}{e_{i-2} e_{i-1}}$ for $2 \leq i \leq m-1$.
A Hurwitz polynomial $H(z)$ associated with $h(z)$ is given by
the $m+1$ positive integers $e_0, e_1, e_2 ,  ... ,e_m$ with 
gcd$(e_0, e_1 , \ldots, e_m)$=1.

\begin{example}
\label{exxx01}
The Salem number $\beta = 1.582347\ldots \in A_2$ of degree six, whose minimal polynomial is $T(x)=z^6-z^4-2z^3-z^2+1$, 
admits the polynomial $P(z)=z^6 + z^5 + z^4 + z^2 + 2z +2$ as expansive polynomial lying over $T$ by
$(z-1) T(z) = z P(z) - P^*(z)$. By the mapping $z\mapsto Z$ and $\omega_2$
(Proposition \ref{PPtoHHtoPP}) we obtain the Hurwitz polynomial 
$H(Z)=4 Z^6 + 5 Z^5 + 29 Z^4 + 10 Z^3 + 14 Z^2 + Z + 1$.
$H(z)$ is represented by the continued fraction of the Hurwitz alternant:
$\displaystyle [f_1/f_2/.../f_6](z) = [1 / 4 / 4 / 5 / \frac{4}{5} / \frac{1}{5}](z)$
with $e_0=e_1=1, e_2 = 4 , e_3 = 16 , e_4 = 320 , e_5 = 4096, e_6 = 64$. 

However, though simple and suggested in \cite{henrici} as one of the basic transformations
of Stieltjes (\cite{stieltjes}, pp J1-J5), 
the coding of the Hurwitz alternants in integers is not practical from a 
numerical viewpoint since we quickly obtain very large integers
$e_i$ for many Salem numbers.
For instance, for the Lehmer number $\theta = 1.176\ldots$ of minimal polynomial
$T(z)= z^{10}+z^9-z^7-z^6-z^5-z^3+z+1$, of degree 10, the two expansive polynomials 
$P_1(z) = z^{10}+2z^9+z^8 +z^2 + 2z+2$ and 
$P_2(z) = z^{10}+z^9+z^8-z^7-z^6-z^5-z^4-z^3 + 2z+2$
lie over $T$, producing $\theta$.
The Hurwitz alternant $h_1$ of $H_1 = \omega_2 (P_1)$, resp.  $h_2$ 
of $H_2 = \omega_2 (P_2)$, is
$$h_1(z)=\frac{10z^4+120z^3+252z^2+120z+10}{9z^5+269z^4+770z^3+434z^2+53z+1},$$
$$\mbox{} {\rm resp.} \qquad 
h_2(z)=\frac{4z^4+112z^3+280z^2+112z+4}{3z^5+261z^4+798z^3+426z^2+47z+1}.$$
For $h_1$, $e_0 = 1, e_1 = 10, e_2 = 410, e_3 = 8320, e_4 \approx
2.272 \, 10^7 , e_5 \approx 1.772 \, 10^{11},$ 
$e_6 \approx 1.944 \, 10^{18}, e_7 \approx
2.594 \, 10^{29} , e_8 \approx 1.414 \, 10^{18} , e_9 \approx 1.151 \, 10^{47} , 
e_{10} \approx 1.528 \, 10^{64}$.
For $h_2$, $e_0 = 1, e_1 = 4, e_2 = 76, e_3 = 2816 , e_4 = 3329024, e_5 = 6335496192,$ 
$e_6 \approx 3.189 \, 10^{16}, e_7 \approx 8.641 \, 10^{15}, e_8 \approx \, 10^{26} , 
e_9 \approx 1.755 \, 10^{41} ,$ 
$e_{10} \approx 3.758 \, 10^{55}$.
\end{example}

{\bf Notation:}  
putting $H(z) = 1 + h_1 z + h_2 z^2 + \ldots + h_m z^m$, a real Hurwitz
polynomial, we denote by
\begin{equation}
\label{omega3}
\omega_3: \rb^m \to (\rb^{+})^m, 
~~\mbox{}^{t}(h_1 , h_2 , \ldots , h_{m}) \mapsto 
\mbox{}^{t}(f_1 , f_2 , \ldots , f_m)
\end{equation}
the mapping which sends the coefficient vector 
of $H$ to the $m$-tuple $(f_j)$ given
by \eqref{fjHANKEL},
where each $f_j$,
strictly positive, belongs to the real field 
$\qb(h_1 , \ldots, h_m)$ generated by the 
coefficients of $H$. 
We consider the $m$-tuple $(f_1,\ldots, f_m)$ as
the set of coordinates of a point of the $m$th-Euclidean space
$\rb^m$, in the canonical basis, generically denoted by $F$.
Conversely (since Salem numbers have even degrees, 
together with the Salem, expansive and Hurwitz polynomials
we consider, we only discuss the case where $m \geq 2$ is even), 
we denote by
\begin{equation}
\label{Omega3}
\Omega_3: \bigl(\rb^{+}\bigr)^{m} \to \rb^m, ~\mbox{}^{t}(f_1, f_2, \ldots , f_m) \mapsto 
\mbox{}^{t}(h_1 , h_2, \ldots , h_m)
\end{equation}
the mapping which sends the $m$-tuple $(f_j)$
of positive real numbers
to the coefficient vector of  
$H(z) = 1 + h_1 z + h_2 z^2 + \ldots + h_m z^m$
by the relations
\begin{equation}
\label{exterior01}
\mbox{for $0 < i \leq \frac{m}{2}$}:\qquad \qquad 
\qquad h_{2i} = \sum_{\{(l_1 , l_2 , \ldots, l_i)\}} ~\prod_{j=1}^{i} f_{l_j}
\end{equation}
where the sum is taken over all $i$-tuples $(l_1 , \ldots, l_i)$ such that 
$l_s \in \{2, \ldots, m\}$ for $s=1, \ldots, i$, and
$l_r - l_s \geq 2$ for all $r > s$,
\begin{equation}
\label{exterior02}
\mbox{for $ 0 \leq i \leq \frac{m}{2} -1$}: 
\qquad \qquad
h_{2i+1}=f_1 \sum_{\{(l_1 , l_2 , \ldots, l_i)\}} ~\prod_{j=1}^{i} f_{l_j}
\end{equation}
where the sum is taken over $i$-tuples $(l_1 , \ldots, l_i)$ such that 
$l_s \in \{3, \ldots, m\}$ for $s=1, \ldots, i$, and
$l_r - l_s \geq 2$ for all $r > s$.

The map $\omega_3 \circ \Omega_3$ is the identity transformation
on $\bigl(\rb^{+}\bigr)^{m}$ and $\Omega_3 \circ \omega_3$ 
is the identity mapping on the set of real Hurwitz
polynomials of degree $m$ 
having constant term equal to $1$.

\subsection{Salem numbers from $m$-terminating continued fractions}
\label{sec:SS3S3}

Let $\beta$ be a Salem number. 
Let $m \geq \deg \beta$ be an even integer.
Let $T(X)$ be a Salem polynomial
vanishing at $\beta$, of degree $m$, having simple roots
such that $T(\pm 1) \neq 0$.
In the sequel, for convenience, we denote
$$
\begin{array}{llcllc}
P = \omega_{1} (T) & {\rm instead ~of} & \eqref{omega1} ; & 
T = \Omega_{1} (P) & {\rm instead ~of} & \eqref{Omega1} ;\\
H = \omega_{2} (P) & {\rm instead ~of} & \eqref{omega2} ; & 
P = \Omega_{2} (H) & {\rm instead ~of} & \eqref{Omega2} ;\\
F = \omega_{3} (H) & {\rm instead ~of} & \eqref{omega3} ; & 
H = \Omega_{3} (F) & {\rm instead ~of} & \eqref{Omega3} ,\\
\mbox{} \hspace{-0.7cm} {\rm and} & & & & & \\
P^{*} = \Omega^{*}_{2} (H) & {\rm instead ~of} & \eqref{Omega2star}, & & &
\end{array}
$$
where $P$ (expansive) is always monic, and $H$ (real Hurwitz) and
$P^{*}$ always have
a constant term equal to $1$. 
The Hurwitz alternant of $H$ is
$h_{P}(z) = h(z)=[f_1 / f_2 / \ldots /f_m](z)$ with $F = \mbox{}^{t}(f_1, f_2, \ldots, f_m)$,
where the subscript ``{\it P}" in $h_{P}(z)$ 
refers to the expansive polynomial $P$.
If $\beta$ belongs to the class $A_q , q \geq 2$, then the Mahler measure 
M$(P)=q$, and 
Theorem \ref{theoremA} amounts to the following:
\begin{equation}
\label{BASIC}
(z-1) T(z) = \bigl( z (\Omega_2 \circ \Omega_3) - (\Omega^{*}_2 \circ \Omega_3)\bigr)(F), 
\end{equation}
\begin{equation}
\label{BASICexpansive}
P(z) = \omega_1 (T)(z) = (\Omega_2 \circ \Omega_3)(F),
\end{equation}
with $L(X) = P(X) - T(X) = \omega_{1} (T)(X) - T(X)$ the reciprocal polynomial 
whose roots, all of modulus $1$,
interlace with those of $T$ on $|z|=1$.

\begin{definition}
\label{interlacingconjugates}
We call {\it interlacing conjugates of $\beta$}, with respect to $T$ and $P$, all on the unit circle, 
the roots of $L$ and the roots of modulus $1$ of $T$.
\end{definition}
 
The set of interlacing conjugates of $\beta$ is the union of
the subcollection of the Galois conjugates of $\beta$ of modulus $1$, and
the subcollection of the roots of $L$ and
the roots of unity which are zeroes of $T$.
There are finitely many interlacing configurations on the unit circle
by Proposition \ref{EqFINITE}.

Once $T$ is fixed, $m \geq 4$ even, with the notations of
\S \ref{sec:SS4.7S2}, denote
$$\mathcal{F}_{T}^{+} := 
\left\{ F = \omega_{3} \circ \omega_{2}(P) \mid
P \in \mathcal{P}_{T}^{+}
\right\}, \, 
\mathcal{F}_{T}^{-} := 
\left\{ F = \omega_{3} \circ \omega_{2}(P) \mid
P \in \mathcal{P}_{T}^{-}
\right\},$$
$$\mathcal{F}_{T} := 
\mathcal{F}_{T}^{+}
\cup 
\mathcal{F}_{T}^{-}
\subset \bigl(\qb_{> 0}\bigr)^m.$$

\noindent
{\it Proof of Theorem \ref{main2}:}

\vspace{0.1cm}

(i) The set $\mathcal{P}_{T}$ of monic expansive polynomials $P$ over
$T$, viewed as a point subset of $\zb^m$ by their coefficients vectors, 
is uniformly discrete: it has a minimal interpoint distance equal to 1.
Considering $\mathcal{P}_{T}$ and 
$\mathcal{F}_{T}$ embedded in $\cb^m$,
the maps $\omega_2$ and $\omega_3$, defined on
$\cb^m$ by \eqref{omega2} and \eqref{omega3} respectively, are analytical.
Invoking the Theorem of the open image for the analytic mapping
$\omega_3 \circ \omega_2$, every open disk centered at a point
$P$ in $\mathcal{P}_{T}$ is sent to an open neighbourhood
of its image point $F = \omega_3 \circ \omega_2 (P)$.
Therefore $\mathcal{F}_{T}$ is a 
union of isolated points 
in the octant $(\rb_{> 0})^m$.

(ii) Suppose that the affine dimension is less than $m$. Then
there would exist an affine hyperplane $y(x_1,\ldots, x_m)=0$
containing $\mathcal{F}_{T}$ in $\rb^m$, and in particular
$\mathcal{F}_{T}^{+}$. The image
of this affine hyperplane by $\Omega_2 \circ \Omega_3$ would be
a hypersurface of $\rb^m$ and the set $\mathcal{P}_{T}^{+}$ 
of monic expansive polynomials $P$ over $T$ would be included in it.
But there is a bijection between $P$ and $\underline{P}$
and the discrete sector $\emph{C}^{+} \cap \zb^{m/2}$ has affine dimension $m/2$.
Contradiction.

(iii) By transport of the internal law $\oplus$ on 
$\mathcal{P}_{T}^{+} \cup \{T\}$, we obtain a semigroup
structure on $\mathcal{F}_{T}^{+} \cup \{\omega_3 \circ \omega_2 (T)\}$,
where the neutral element is $\omega_3 \circ \omega_2 (T)$. Let us 
show that
$$\displaystyle \omega_3 \circ \omega_2 (T) (z) =\frac{h_1+h_3z+h_5z^2+..}{h_0+h_2z+h_4z^2+..}= 0.$$
In fact,
$$
\begin{aligned}
\displaystyle h_{2i+1} &= \sum_{j=0} ^{m} \sum_{k=0}^{2i+1} (-1)^{j-k} \binom{j}{k} \binom{m-j}{2i+1-k} t_j \\
&=\sum_{j=0} ^{\frac{m}{2}-1}  \sum_{k=0}^{2i+1} (-1)^{j-k} \binom{j}{k} \binom{m-j}{2i+1-k} t_j\\
&+ \sum_{j=\frac{m}{2}} ^{m}  \sum_{k=0}^{2i+1} (-1)^{j-k} \binom{j}{k} \binom{m-j}{2i+1-k} t_{m-j}\\
&= \sum_{j=0} ^{\frac{m}{2}-1}  \sum_{k=0}^{2i+1} (-1)^{j-k} \binom{j}{k} \binom{m-j}{2i+1-k} t_j\\ 
&+ \sum_{j=0} ^{\frac{m}{2}-1}  \sum_{l=0}^{2i+1} (-1)^{m-j-2i-1+l} \binom{m-j}{2i+1-l} \binom{j}{l} t_{j}\\
&= \sum_{j=0} ^{\frac{m}{2}-1}  \sum_{k=0}^{2i+1} ((-1)^{j-k}+(-1)^{-j-1 +k})  \binom{j}{k} \binom{m-j}{2i+1-k} t_j\\
 &=0.\\
\end{aligned}
$$
Since $T(-1) \neq 0$, the coefficient $h_0$ is not equal to $0$.
The image of $T$ by $\omega_2$ is of type $\mbox{}^{t}(h_0 , 0, h_2 , 0, h_4 , 0, \ldots, 0 , h_m)/h_0$.
It admits a Hurwitz alternant $h$ equal to $[0/0/\ldots/0](z)$, where $0$ occurs $m$ times,
corresponding to the point $F=0$ in $\rb^m$. 
Indeed, all the coefficients
$f_j , j=1, \ldots, m,$ are equal to $0$, by \eqref{fjHANKEL}, 
since the top rows of the 
determinants $D_j$ are all made of zeros.

(iv) The real Hurwitz polynomials $H$ obtained as images of elements
of $\mathcal{P}_{T}$ by $\omega_2$ never vanish at $z=-1$.
The image under $\omega_3$ of the hyperplane of $\rb^m$ 
of equation $\{\mbox{}^{t}(h_1, h_2, \ldots, h_m) \in \rb^m \mid
1 - h_1 + h_2 - \ldots + h_m = 0\}$ is an hypersurface
which does not contain any point $F$ of $\mathcal{F}_{T}$.
$\Box$

Once  $T$ is fixed, and $P$ chosen over $T$,
the geometrical coding
\begin{equation}
\label{correspondance}
( \, \beta, T, P\,) 
\quad \longleftrightarrow \quad F = \omega_{3} \circ \omega_{2} \circ \omega_{1}(T) = 
\left(\begin{array}{c}
f_1\\
\vdots
\\ f_m
\end{array}\right) \in (\qb_{>0})^{m}
\end{equation}
shows that the Stieltjes continued fraction  $F$ 
entirely controls $\beta$ and simultaneously all the interlacing conjugates
of $\beta$, by \eqref{BASIC} and \eqref{BASICexpansive}. 
When $P$ runs over the complete set of monic expansive 
polynomials over $T$, $F$ runs over $\mathcal{F}_{T}$.
If $F_1$ and $F_2 \neq F_1$ are in $\mathcal{F}_{T}$, they
give rise to equivalent codings of the {\it same} 
Salem number $\beta$.

As for the commutative semigroup law on Hurwitz alternants,
also denoted by $\oplus$,
Theorem \ref{main2} implies that the following diagram
is commutative, for
$P, P^{\dag} \in \mathcal{P}_{T}^{+}$:
$$
\begin{array}{cccc}
(P, P^{\dag}) & \to & P \oplus P^{\dag} & := P + P^{\dag} - T\\
 & & & \\
\downarrow & & \downarrow & \\
 & & & \\
(h_P , h_{P^{\dag}}) & \to & h_{P} \oplus h_{P^{\dag}} & := h_{P \oplus P^{\dag}}
\end{array}
$$ 

\noindent
\section*{Appendix: ~Proof of Lemma \ref{intervals_nonexistence}.}

Let $a$ and $b > a > 1$ be two real numbers such that the open interval
$(a,b)$ contains no Salem number. If there exist Salem numbers
arbitrarily close to $1$, i.e. if the adherence set $\overline{ \rm T}$ contains $1$,
then there would exist 
$\tau_0 ={\rm M}(\tau_0) \in $ T such that:
$0 \neq \log (\tau_0) < \log b -\log a$.
The greatest integer $s = \lfloor \frac{\log a}{\log(\tau_0)} \rfloor > 0$ such that
$s \log(\tau_0) \leq \log a$ satisfies: 
$\log a < (s+1) \log(\tau_0) < \log b$. Indeed,
if $(s+1) \log(\tau_0) > \log b$, then we would have:
$s \log(\tau_0) > \log b - \log(\tau_0) > \log a$. Contradiction.
Then $a < \tau_{0}^{s+1} < b$. But
$\tau_{0}^{s+1}$ is a Salem number since $\tau_{0} \in $ T (Salem \cite{salem0} p. 106).
Contradiction.

\section*{Acknowledgements}

The authors are indebted to M.-J. Bertin for her interest and valuable comments and discussions.

\end{document}